\documentclass[12pt,a4paper,twoside]{amsart}
\usepackage{charter}
\usepackage{mathrsfs}
\usepackage{enumitem}
\usepackage{rotating}
\usepackage{multirow} 
\usepackage{array}  
\usepackage{amsmath, amsfonts, amssymb, amsthm, amscd}
\usepackage{etoolbox} 
\usepackage{stackrel}
\usepackage[all]{xy}
\usepackage{fancyhdr}
\usepackage{tikz}
\usetikzlibrary{decorations.markings}
\usepackage{hyperref}
\usepackage{color}

\usepackage[mathscr]{eucal}

\begin{document}
\newtheorem{theorem}{Theorem}[section]
\newtheorem{theoremast}[theorem]{{}*Theorem}
\newtheorem{lemma}[theorem]{Lemma}
\newtheorem{remark}[theorem]{Remark}
\newtheorem{definition}[theorem]{Definition}
\newtheorem{lemdef}[theorem]{Lemma-Definition}

\newtheorem{lemmadef}[theorem]{Lemma-definition}
\newtheorem{proposition}[theorem]{Proposition}
\newtheorem{corollary}[theorem]{Corollary}

\theoremstyle{definition}
\newtheorem{example}[theorem]{Example}
\newtheorem{examples}[theorem]{Examples}
\newtheorem{none}[theorem]{}
\theoremstyle{remark}
\newtheorem{caution}[theorem]{Caution}
\newtheorem{remarks}[theorem]{Remarks}
\newtheorem{question}[theorem]{Question}
\newtheorem{problem}{Problem}

\newtheorem{exercise}{Exercise}

\renewcommand{\bar}{\overline}
\def\C{\mathbb{C}}
\def\R{\mathbb{R}}

\def\T{\mathrm{T}}
\def\X{\mathcal{X}}
\def\U{\mathcal{U}}
\def\P{\Phi}
\def\M{\mathcal{M}}
\def\Z{\mathcal{Z}_{m}}
\def\ZZ{\mathcal{Z}_{m}^{H}}
\def\d{\partial}
\def\cP{\mathscr P}
\def\bB{\mathbb B}
\def\TT{\mathcal{T}_{m}^{H}}
\def\GD{\Gamma \backslash D}
\def\mf{\mathfrak}
\def\ad{\mathrm{ad}\,}
\def\Ad{\mathrm{Ad}\,}
\def\mI{\mathscr{I}}
\def\mU{\mathscr{U}}
\def\c{\mathrm{c}}
\def\nc{\mathrm{nc}}
\def\nca{{\mathrm{nc},1}}
\def\ncb{{\mathrm{nc},2}}
\def\i{\sqrt{-1}}

\title{Penrose transformation on flag domains}
\author{Kefeng Liu}
\address{Mathematical Sciences Research Center, Chongqing University of Technology, Chongqing 400054, China; \newline
Shanghai Institute for Mathematics and Interdisciplinary Sciences, Shanghai 200433, China}
\email{liu@math.ucla.edu}

\author{Yang Shen}
\address{Mathematical Sciences Research Center, Chongqing University of Technology, Chongqing 400054, China}
\email{syliuguang2007@163.com}

\begin{abstract}
%
Building on our recent work, we construct the Penrose transformations of the cohomology groups of homogeneous line bundles on flag domains $D = G_\R / T$, where $G_\R$ is of Hermitian type. 
We provide sufficient conditions for the injectivity of the Penrose transformation and identify conditions under which the Penrose transformation of the automorphic cohomology groups on compact quotients of flag domains is an isomorphism.
Finally, we prove that the higher automorphic cohomology groups of certain homogeneous line bundles are isomorphic to the groups of automorphic forms on the Hermitian symmetric domain, and we apply this result to the cup products of the automorphic cohomology groups.
\end{abstract}

\maketitle

\tableofcontents






\setcounter{section}{-1}
\section{Introduction}
As noted in the preface of \cite{BE}, the Penrose transformation was originally introduced by Roger Penrose in the 1960s as an isomorphism between a sheaf cohomology group on a region of projective space and the solutions of a zero-rest-mass field equation on a region of spacetime in mathematical physics. Since its introduction, numerous generalizations of the Penrose transformation have been studied in both mathematics and physics. See, for example \cite{BE}, for a detailed discussion on Penrose transformations on flag varieties.

In recent works \cite{C1}, \cite{C2}, and \cite{C3} by Carayol, as well as in \cite{GGK} and \cite{GGK14} by Green, Griffiths, and Kerr, the Penrose transformation has been studied on flag domains for $SU(2,1)$ and $Sp(4)$. In the context of Hodge theory and representation theory, the Penrose transformation on flag domains investigates the relationship between the cohomology groups of homogeneous line bundles on a classical flag domain $D'$ and those on the non-classical flag domain $D$  that is diffeomorphic to $D'$. Additionally, it also studies the relations between the corresponding automorphic cohomology groups on the compact quotients $X$ and $X'$ of $D$ and $D'$ respectively.

As mentioned on Page 223 and Page 233 of \cite{GGK}, little is known about the Penrose transformation on flag domains, except for some special cases.

In our recent work \cite{LiuShen24}, we proved that any non-classical flag domain $D = G_\R / V$ with $G_\R$ of Hermitian type is diffeomorphic to a classical flag domain $D'$, where $G_\R$ is a semisimple real Lie group and $V \subset G_\R$ is a compact subgroup containing a compact Cartan subgroup $T$. This represents the most general result on this topic, since any real Lie group $G_\R$ acting on a classical flag domain must be of Hermitian type. Furthermore, in \cite{LiuShen24}, we also showed that the compact, smooth quotient $X$ of $D$ admits no $\partial \bar{\partial}$-structure compatible with the complex structure induced from $D$. This result is obtained by proving a conjecture of Green--Griffiths--Kerr concerning the vanishing of the zeroth automorphic cohomology groups of any non-trivial homogeneous vector bundles on $X$.

In contrast, it was proved in \cite{GS} that the compact quotient $X'$ of the classical flag domain $D'$ is a projective manifold, and thus, it has an arithmetic structure on its automorphic cohomology group. Once the Penrose transformation is shown to be an isomorphism, the automorphic cohomology group on $X$ will inherit the arithmetic structure from that on $X'$. 

Therefore, in conjunction with our work \cite{LiuShen24}, the Penrose transformation on flag domains reveals unexpected useful properties. For further important applications of the Penrose transformation, we refer the reader to  Lecture 9 in \cite{GGK} and \cite{GGK14}.

In this paper, we develop the Penrose transformation on general non-classical flag domains of the form $D = G_\R / T$ with $G_\R$ of Hermitian type, within the framework as developed by Green, Griffiths and Kerr in \cite{GGK} and \cite{GGK14}.
Our proof is based on overcoming several technical difficulties by combining methods from deformation theory, geometric realizations of the representations of compact and non-compact real Lie groups, and geometry of homogenous manifolds.

Now let us introduce the main results of this paper.

Let $D=G_{\R}/T$ be a non-classical flag domain with $G_{\R}$ of Hermitian type, and $D'$ be the classical one which is diffeomorphic to $D$.
The complex structures of $D$ and $D'$ are given respectively by
\begin{eqnarray}
  \mathrm{T}^{1,0}_o D &\cong& \mf n_-=\mf k_-\oplus \mf p_-^1\oplus \mf p_-^2 \label{intr complex D}; \\
  \mathrm{T}^{1,0}_o D' &\cong& \mf n'_-=\mf k_-\oplus \mf p_+^1\oplus \mf p_-^2.\label{intr complex D'}
\end{eqnarray}
They only differ in the directions $\mf p_\pm ^1\subset \mf p$, where $\mf p_{+}^{1}=\bar{\mf p_{-}^{1}}$. Please see Section \ref{Pre} for the notations of the Lie algebras and the description of the complex structures of flag domains in terms of Lie algebras.

Let $\Delta$ be the root system for the complex Lie algebra $\mf g$ with respect to the Cartan subalgebra $\mf h$, where $\mf g$ and $\mf h$ are the complexifications of the Lie algebras $\mf g_{0}$ and $\mf h_{0}$ of $G_{\R}$ and $T$ respectively. Then we have the root space decomposition
$$\mf g=\mf h \oplus \sum_{\alpha\in \Delta}\mf g_{\alpha}.$$
Let $\Delta_{+}$ and $\Delta'_{+}$ be the set of positive roots such that 
$$\mf g_{\alpha}\subset \mf n_{-},\, \mf g_{\alpha'}\subset \mf n'_{-},\, \forall\, \alpha\in \Delta_{+},\, \alpha'\in \Delta'_{+}.$$
The decomposition \eqref{intr complex D} induces the union
$$\Delta_{+}=\Delta_{+}^{\c}\cup \Delta_{+}^{\nc}=\Delta_{+}^{\c}\cup \Delta_{+}^{\nca}\cup\Delta_{+}^{\ncb}$$
such that the corresponding space $\mf g_{\alpha}$ lies in $\mf k_-, \mf p_-^1, \mf p_-^2$ respectively. Then the decomposition \eqref{intr complex D'} implies that
$$\Delta'_{+}={\Delta'}_{+}^{\c}\cup {\Delta'}_{+}^{\nc}=\Delta_{+}^{\c}\cup \left(-\Delta_{+}^{\nca}\right)\cup\Delta_{+}^{\ncb}.$$
In this paper, $\rho_{\#}$ ($\rho'_{\#}$ resp.) denotes the half of the sum of the roots in $\Delta_{+}^{\#}$ (${\Delta'}_{+}^{\#}$ resp.), where $``\#''$ can be taken to be  one of the symbols $``\emptyset'',``\c'',``\nc'',``\nca'',``\ncb''$.

From Lecture 7 in \cite{GGK}, we know that there is an associated space $\mathscr W$ which is called the correspondence space for $D$ and $D'$. Moreover $\mathscr W$ is Stein with the basic diagram
\begin{equation}\label{intr basic diagram}
\xymatrix{
&\mathscr W\ar[ld]^-{\pi}\ar[rd]_-{\pi'}&\\
D & & D'
}\end{equation}
where the fibers of $\pi$ and $\pi'$ are contractible. 
Then, one can apply the EGW Theorem in \cite{EGW} to the homogeneous line bundles $L_\mu$ and $L_{\mu'}$ on $D$ and $D'$ respectively to obtain
\begin{eqnarray}
  H^*(D,L_{\mu}) &\cong& H^*_{DR}(\mathscr W,\Omega_{\pi}^\bullet(L_{\mu})),\label{intr 1} \\
  H^*(D',L_{\mu'}) &\cong& H^*_{DR}(\mathscr W,\Omega_{\pi'}^\bullet(L_{\mu'})),\label{intr 2}
\end{eqnarray}
where $$\Omega^\bullet_{\pi}(L_\mu)=\Omega^\bullet_{\mathscr W}/\pi^*\Omega^\bullet_D \otimes_{\mathcal{O}_\mathscr{W}} \pi^*L_\mu$$ with the induced differential $d_\pi$, and the de Rham cohomology groups $$H^*_{DR}(\mathscr W,\Omega_{\pi}^\bullet(L_{\mu}))$$ are the cohomology groups of the sequence of the global sections $$\Gamma(\mathscr W,\Omega_{\pi}^\bullet(L_{\mu});\,d_\pi).$$ The notations for $\pi'$ are similar. 

From \eqref{intr 1} and \eqref{intr 2}, the cohomology groups $H^*(D,L_{\mu})$ and $H^*(D',L_{\mu'})$ can be related on $\mathscr W$. More precisely we have the following definition of Penrose transformation on flag domains.


\begin{lemdef}\label{intr Pen defn}
Let $\mu$ and $\mu'$ be two weights such that $\mu+\rho=\mu'+\rho'$.
Then the Penrose transformation $$\mathscr P:\, H^0(D',L_{\mu'})\to H^q(D,L_{\mu})$$ is defined by
\begin{equation}\label{intr PenD}
  \xymatrix{
  H^0(D',L_{\mu'}) \ar[d]^-{\cong}\ar[rr]^-{\mathscr P}&& H^q(D,L_{\mu})\ar[d]^-{\cong}\\
  H^0_{DR}(\mathscr W,\Omega_{\pi'}^\bullet(L_{\mu'}))\ar[rr]^-{\omega^\nca}&&H^q_{DR}(\mathscr W,\Omega_{\pi}^\bullet(L_{\mu}))
  }
\end{equation}
where $q=\dim_\C \mf p_+^1=\#\Delta_{+}^{\nca}$ with $\Delta_{+}^{\nca}=\{\beta_{1},\cdots,\beta_{q}\}$ and the map $\omega^\nca$ denotes multiplication by $$\omega^\nca=\omega^{-\beta_1}\wedge \cdots \wedge \omega^{-\beta_q}$$ which is considered as a left-invariant $d_\pi$-closed differential $q$-form with values in $L_{-2\rho_\nca}$ on $\mathscr W$.
\end{lemdef}
See Section \ref{Pre} for the precise definition of $\omega^\nca$.  The following theorem is the main result of this paper.

\begin{theorem}[Theorem \ref{mainD}]\label{intr mainD}
Let the notations and assumptions be as Lemma-Definition \ref{intr Pen defn}.
If for any $\beta\in\Delta_+^\nca$ there exists an $\alpha\in \Delta_+^\c$ such that 
\begin{equation}\label{intr Pen inj conditions}
  (\alpha,\,\mu'-\beta)<0,
\end{equation}
then the Penrose transformation $$\mathscr P:\, H^0(D',L_{\mu'})\to H^q(D,L_{\mu})$$ given by \eqref{intr PenD} is injective.
\end{theorem}

The main idea of proving Theorem \ref{intr mainD} can be sketched as follows. Let $[F]$ be the corresponding cohomological class of $F \in \Gamma(\mathscr{W}, \Omega_{\pi'}^\bullet(L_{\mu'}))$, such that $[F \omega^\nca] = 0$. This means that there exists $\Psi \in \Gamma(\mathscr{W}, \Omega_{\pi}^{q-1}(L_{\mu}))$ such that $F \omega^\nca = d_\pi \Psi$. We need to show that $F = 0$.

In the special cases of $SU(2,1)$ and $Sp(4)$ where $q = 1$ and $\Psi$ is a global section of $L_{\mu}$ on $\mathscr{W}$, there is a natural space $\mathscr{J}$ and a holomorphic submersion $\mathscr{W} \to \mathscr{J}$, so that $\Psi$ descends to a global section in $\Gamma(\mathscr{J}, L_{\mu})$. Moreover, $\mathscr{J}$ is covered by the cycles isomorphic to $\mathbb{P}^1$. Hence, for the weight $\mu$ in a certain range, one has the vanishing theorem 
$\Gamma(\mathbb{P}^1, L_{\mu}|_{\mathbb{P}^1}) = 0$, 
which implies that $\Psi \in \Gamma(\mathscr{W}, L_{\mu}) = 0$. This proves that $F \omega^\nca = d_\pi \Psi = 0$, and hence $F = 0$. Please see \cite{GGK}, Lecture 8, and \cite{GGK14} for details.

In contrast to the cases of $SU(2,1)$ and $Sp(4)$, in general we do not have the natural space similar to $\mathscr{J}$,  and when $q =\#\Delta_+^\nca\geq 2$,  $\Psi$ may involve differential forms from $\mathfrak{k}_+^*$ and $(\mathfrak{p}_+^2)^*$. 
To overcome the difficulty we turn to study the incidence space
$$ \mathscr{I} = \{(x,u) \in D \times \mathscr{U} : x \in Z_u\} $$
with a holomorphic submersion $\mathscr{W} \to \mathscr{I}$ and prove the vanishing theorem on an open subset of $\mathscr{I}^o \subset \mathscr{I}$ in Theorem \ref{vanshing on I prop}. Here, $\mathscr{U}$ is the cycle space for the non-classical flag domain $D$. In the proof of Theorem \ref{vanshing on I prop}, we use crucially the upper semi-continuity of the dimensions of a differential family of vector bundles proved by Kodaira and Spencer in \cite{KS3}.

Let $\Psi'$ be the projection of $\Psi$ onto $\wedge^{q-1} (\mathfrak{p}_+^1)^*$. In the proof of Theorem \ref{mainD}, we use the locally trivial fibration structure $B_K/T_\C \times V_\nu \times U$ of $\mathscr{W}$, where $\{V_\nu \subset K_\C/B_K\}_\nu $ is an open cover on which the homogenous vector bundles are trivialized and $U\subset \mathscr U$ is an open subset containing the base cycle $u_o\cong Z_o$ of $D$. Then we define $\Xi$ from $\Psi'$ such that $\Xi$ is constant on $B_K/T_\C$, which is the fiber of $\mathscr{W} \to \mathscr{I}$. Hence $\Xi$ descends to a section on an open subset $\mathscr{I}^o=K_\C/B_K\times U$ of $\mathscr I$ containing the base cycle $K_\C/B_K\times \{u_o\}$, and $\Xi$ is zero from the vanishing theorem on $\mathscr{I}^o$. Hence $\Xi=0$ on the open subset $(K_\C / T_\C)\times U$ of $\mathscr W$, and 
$$F \omega^\nca |_{(K_\C / T_\C) \times U}=\partial_{\mf p_+^1} \Xi=0,$$
where $\partial_{\mf p_+^1}$ denotes the differentials of the coefficients of $\Xi$ in the directions corresponding to $\mf p_+^1$.
This implies that $F=0$ on the open subset $(K_\C / T_\C) \times U$ of $\mathscr W$, and $[F ]= 0$ as the cohomological class in $\Gamma(\mathscr{W}, \Omega_{\pi'}^\bullet(L_{\mu'}))$.
This is the main idea of the proof of Theorem \ref{intr mainD}.

Next, we introduce the Penrose transformation on the compact quotients of flag domains.

Let $\Gamma\subset G_\R$ be a discrete, co-compact and neat subgroup. Then
$X=\Gamma\backslash D $ and $X'=\Gamma\backslash D'$ are compact complex manifolds, and $\mathscr W_\Gamma=\Gamma\backslash \mathscr W$ is a complex manifold which is also Stein. 
Then the basic diagram \eqref{intr basic diagram} becomes
\begin{equation}\label{intr basic diagram compact}
\xymatrix{
&\mathscr W_\Gamma\ar[ld]^-{\pi}\ar[rd]_-{\pi'}&\\
X & & X'.
}\end{equation}

\begin{lemdef}\label{intr Pen defn compact}
Let the notations and assumptions be as Lemma-Definition \ref{intr Pen defn}.
Let $X=\Gamma\backslash D $ and $X'=\Gamma\backslash D'$ be the corresponding quotients by a discrete, co-compact and neat subgroup $\Gamma\subset G_\R$.
Then the Penrose transformation on automorphic cohomology groups
$$\mathscr P:\, H^0(X',L_{\mu'})\to H^q(X,L_{\mu})$$ is defined by
\begin{equation}\label{intr PenD compact}
  \xymatrix{
  H^0(X',L_{\mu'}) \ar[d]^-{\cong}\ar[rr]^-{\mathscr P}&& H^q(X,L_{\mu})\ar[d]^-{\cong}\\
  H^0_{DR}(\mathscr W_\Gamma,\Omega_{\pi'}^\bullet(L_{\mu'}))\ar[rr]^-{\omega^\nca_\Gamma}&&H^q_{DR}(\mathscr W_\Gamma,\Omega_{\pi}^\bullet(L_{\mu})),
  }
\end{equation}
where $\omega^\nca_\Gamma$ is the multiplication by the differential form $\omega^\nca_\Gamma$ on $\mathscr W_\Gamma$ to which the left-invariant differential form $\omega^\nca$ given in Lemma-Definition \ref{intr Pen defn} descends.
\end{lemdef}

Applying the proof of Theorem \ref{intr mainD}, we have the following theorem on the injectivity of the Penrose transformation on automorphic cohomology groups.

\begin{theorem}[Theorem \ref{mainD compact1}]\label{intr mainD compact1}
Let the notations be as Lemma-Definition \ref{intr Pen defn compact}. 
If for any $\beta\in\Delta_+^\nca$ there exists an $\alpha\in \Delta_+^\c$ such that 
\begin{equation}
  (\alpha,\,\mu'-\beta)<0,\tag{\ref{intr Pen inj conditions}}
\end{equation}
then the Penrose transformation on automorphic cohomology groups $$\mathscr P:\, H^0(X',L_{\mu'})\to H^q(X,L_{\mu})$$ given by Lemma-Definition \ref{intr Pen defn compact} is injective.
\end{theorem}

Hence, in order for the Penrose transformation on the compact quotients to be an isomorphism, we only need to check the dimensions of the both sides.
Thanks to Theorem 2.4 in \cite{Wi2}, we know that 
$$\dim_\C H^0(X',L_{\mu'}) = \dim_\C H^q(X,L_{\mu})$$ 
depends only on the number determined by the weight $$\mu+\rho = \mu'+\rho'$$ provided that $\mu+\rho = \mu'+\rho'$ is regular in some Weyl chamber and satisfies Property \textbf{W} of Williams in Lemma \ref{W lemma}, which is equivalent to conditions \eqref{intr PW} below.

\begin{theorem}[Theorem \ref{mainD compact2}]\label{intr mainD compact2}
Let the notations be as Lemma-Definition \ref{intr Pen defn compact}.
Let $\mu$ and $\mu'$ be two weights such that $\mu+\rho=\mu'+\rho'$ is regular and $\mu'+\rho'$ lies in the Weyl chamber determined by the set of roots,
\begin{equation}\label{intr chamber}
  \Delta_+^\c \cup \Delta_+^\nca\cup\left(-\Delta_+^\ncb \right)=\Delta_+^\c \cup\left(-{\Delta'}_+^\nc \right).
\end{equation}
If  conditions \eqref{intr Pen inj conditions} are satisfied and moreover, 
\begin{equation}\label{intr PW}
 (\mu'+2\rho'_\nc,\, -{\Delta'}_+^\nc )>0,
\end{equation}
then the Penrose transformation on automorphic cohomology groups $$\mathscr P:\, H^0(X',L_{\mu'})\to H^q(X,L_{\mu})$$ is an isomorphism.
\end{theorem}

As an application of the main results of this paper, we provide an important example of homogeneous line bundles derived from the canonical bundles on Hermitian symmetric domains, which satisfy the conditions in Theorem \ref{intr mainD compact2}. This example has significant applications in arithmetic geometry and number theory.

Let
$$\mu'_c=-\sum_{\beta \in {\Delta'}_+^\nc}\beta= 2\rho_\nca-2\rho_\ncb.$$ 
and $k_0$ the maximal positive integer such that $\mu'_{c0}\triangleq \mu'_c/k_0$ is still a weight.
Let $$\mu_k'=k\mu'_{c0}=\frac{k}{k_0}(2\rho_\nca-2\rho_\ncb),\,\mu_k=\mu_k'-2\rho_\nca.$$ 
Then
$$L_{\mu_k'}=\omega_{\mathbb B}^{\otimes k/k_0}\to D',\, L_{\mu_k}\to D$$
are two line bundles whose cohomology groups are Penrose related, where $\omega_{\mathbb B}$ is
the pull-back of the canonical bundle on $G_\R/K\cong \mathbb B$ via the holomorphic projection map
$$p':\, D'\to G_\R/K.$$

\begin{theorem}[Theorem \ref{automorphic theorem}]
Let the notations be as above. Then there exists a positive integer $N$ such that for $k\ge N$, the Penrose transformation 
\begin{equation}\label{intr PT automorphic form}
 \mathscr P:\, H^0(X',\omega_{\mathbb B}^{\otimes k/k_0})\to H^q(X,L_{\mu_k})
\end{equation}
is an isomorphism. Therefore $H^q(X,L_{\mu_k})$ is isomorphic to the group $$H^0(\Gamma\backslash \mathbb B,\omega_{\mathbb B}^{\otimes k/k_0})$$ of automorphic forms on $\mathbb B$.
\end{theorem}

Finally we consider the cup-products of the automorphic cohomology groups on $D$, which are Penrose related to automorphic cohomology groups on $\mathbb B$ and $\bar{\mathbb B}$, with the target being a TDLDS (totally degenerate limits of discrete series).

\begin{theorem}
Let $D=G_{\R}/T$ be a non-classical flag domain with $G_{\R}$ of Hermitian type. Let $X=\Gamma\backslash D$ be a smooth and compact quotient of $D$. Then

(1) if $G_{\C}=SU(s+1,s)$, then there exists a positive integer $N$ such that for $k\ge N$, 
we have a homomorphism of groups
\begin{equation*}
H^0(\Gamma\backslash\mathbb B,\omega_{\mathbb B}^{\otimes k/k_0})\times H^0(\Gamma\backslash \bar{\mathbb B}, \omega_{\bar{\mathbb B}}^{\otimes k/k_0}) \to H^{d}(X,L_{-\rho})^*;
\end{equation*}
and

(2) if $G_{\C}=Sp(6,\C)$, then there exist $\beta\in \Delta_{+}^{\nc}$ and a positive integer $N$ such that for $k\ge N$, 
we have a homomorphism of groups
\begin{equation*}
H^0(\Gamma\backslash\mathbb B,\omega_{\mathbb B}^{\otimes k/k_0})\times H^0(\Gamma\backslash \bar{\mathbb B}, \omega_{\bar{\mathbb B}}^{\otimes k/k_0}\otimes L_{\beta}) \to H^{d}(X,L_{-\rho})^*.
\end{equation*}
\end{theorem}

The paper is organized as follows. In Section \ref{Pre}, we introduce the basic notations for the geometry of flag domains and review our results on new complex structures for non-classical flag domains. In Section \ref{Corres}, we introduce the correspondence space and the incidence variety for flag domains, and relate the cohomology groups of homogeneous line bundles on flag domains to the de Rham cohomology groups on the correspondence space. An important vanishing theorem is proved in this section, which is crucial to the
proof of the injectivity of the Penrose transformation. In Section \ref{section Pen}, we define the Penrose transformation on flag domains and establish its injectivity under certain conditions. In Section \ref{section Pen compact}, we introduce the Penrose transformation on the compact quotients of flag domains and prove that it is an isomorphism under the conditions given in Section \ref{section Pen} and the additional conditions induced by Property \textbf{W} of Williams. We also apply our results to the groups of automorphic forms on Hermitian symmetric domains, which can be identified with the higher automorphic cohomology groups of certain line bundles on the compact quotient of the non-classical flag domain. In Section \ref{cup-product section}, we apply our results to study the cup products of automorphic cohomology groups on $ D $, with the target being a TDLDS. In Section \ref{examples}, we present examples from the classical groups $ SU(r,s) $ and $ Sp(2n,\C) $ to illustrate the results proved in the previous sections.


\section{Preliminaries}\label{Pre}
In this section, we introduce the basic notions of flag domains and fix the notations in this paper. Additionally we recall our recent work \cite{LiuShen24} on the new complex structures of non-classical flag domains, which is the foundation of this paper. References for this section are \cite{FHW}, \cite{Hel} and \cite{knapp}.

Let $G$ be a reductive $\mathbb Q$-algebraic group. Let $G_\R$ and $G_\C$ be the associated real and complex connected Lie groups respectively. 
A {flag variety} $\check D =G_\C/B$ is a complex homogenous manifold with $B\subset G_\C$ a Borel subgroup. A {flag domain} $D=G_\R/T$ is an open $G_\R$-orbit in $\check D$, where the isotropy subgroup $T=G_\R\cap B$ is a compact Cartan subgroup of $G_\R$.

Let $K$ be the maximal compact subgroup of $G_\R$ containing the compact Cartan subgroup $T$. The real Lie group $G_\R$ is said to be {of Hermitian type} if the Riemannian symmetric space $G_\R/K$ has a $G_\R$-invariant complex structure. The flag domain $D$ is called {classical} if $G_\R$ is of Hermitian type and the projection map
$$p:\, D\to G_\R/K$$
is a holomorphic map between complex manifolds. Otherwise, $D$ is called {non-classical}.

Hence non-classical flag domains include the cases that $G_\R$ is not of Hermitian type and that $G_\R$ is of Hermitian type but the projection map $p$ is not holomorphic.
In this paper, we mainly consider the non-classical flag domain $D$ with $G_\R$ of Hermitian type.

The complex structures of $D$ and $G_\R/K$ can be described by the Lie algebras of the Lie groups mentioned above.
For this, we fix some notations of the Lie algebras:
\begin{itemize}
  \item Let $\mf h_0\subset \mf k_0 \subset \mf g_0$ be the Lie algebras of $T\subset K\subset G_\R$ respectively;
  \item Let $\mf b\subset \mf g$ be the Lie algebras of $B\subset G_\C$ respectively. Then $\mf g =\mf g_0\otimes_\R \C$ is the complexification of $\mf g_0$ and $\mf h_0=\mf b\cap \mf g_0$;
  \item Let $\mf n_-\subset\mf g$ be the subalgebra such that 
        $$\mathrm{T}^{1,0}_o\check D=\mathrm{T}^{1,0}_o D=\mf g/\mf b\cong \mf n_-$$
        where $o$ is the base point in $D$. 
  \item Let $\mf h\subset \mf k$ be the complexications of $\mf h_0\subset \mf k_0$ respectively. Then $\mf h\subset \mf g$ is the Cartan subalgebra, and there exist subspaces $\mf p_0\subset \mf g_0$ and $\mf p=\mf p_0\otimes_\R \C \subset \mf g$ such that 
      \begin{eqnarray}
        &&\mf g_\#=\mf k_\# \oplus \mf p_\#,\, \text{with} \nonumber \\
        && [\mf k_\#,\mf k_\#]\subset \mf k_\#,\, [\mf k_\#,\mf p_\#]\subset \mf p_\#,\, [\mf p_\#,\mf p_\#]\subset \mf k_\# \label{Cartan relation}
      \end{eqnarray}
      are Cartan decompositions on $\mf g_0$ and $\mf g$, where $\#$ is $0$ or $\emptyset$;
      \item Let $\mf n_+=\bar{\mf n_-}$ and $\mf p_\pm =\mf p\cap \mf n_\pm$, $\mf k_\pm =\mf k\cap \mf n_\pm$.
\end{itemize}

We also denote by the complex Lie subgroups $T_\C \subset K_\C \subset G_\C$ corresponding to the complex Lie subalgebras $\mf h\subset \mf k\subset \mf g$.

Since $\mf h \subset \mf g$ is the Cartan subalgebra, there exists a decomposition 
\begin{equation}\label{CartandecompC}
  \mf g =\mf h \oplus \bigoplus_{\alpha \in \Delta}\mf g_\alpha, 
\end{equation}
where $\Delta\subset \sqrt{-1}\mf h_0^*$ is the root system of $\mf g$ with respect to $\mf h$, and
$$\mf g_\alpha =\{X\in \mf g:\, [H,X]=\alpha(H)X, \forall\, h\in \mf h\}$$ is the root space which is one-dimensional with basis $e_\alpha$ for any $\alpha\in \Delta$.

Note that $e_\alpha$, $\alpha\in \Delta$, can be considered as left-invariant vector fields on $G_\C$.
Let $\{\omega^\alpha:\, \alpha\in \Delta\}$ be the dual of $\{e_\alpha:\, \alpha\in \Delta\}$. Then $\omega^\alpha$, $\alpha\in \Delta$, can be considered as left-invariant differential one-forms on $G_\C$.

Now we fix some notations for the root systems:
\begin{itemize}
  \item Since $B\subset G_\C$ is Borel subgroup, there exists a set $\Delta_+$ of positive roots such that 
      $${\mf b} = \mf h \oplus \bigoplus_{\alpha \in \Delta_+}\mf g_{-\alpha},\, \mf n_-=\bigoplus_{\alpha \in \Delta_+}\mf g_{\alpha};$$
  \item Let $\Delta^c,\Delta^\nc \subset \Delta$ be the sets of compact roots and noncompact roots respectively such that
      $$\mf k=\mf h \oplus \bigoplus_{\alpha \in \Delta^{\mathrm{c}}}\mf g_\alpha,\, \mf p = \bigoplus_{\beta \in \Delta^{\mathrm{nc}}}\mf g_\beta .$$
  \item Let $\Delta_+^\c =\Delta_+\cap \Delta^c$ and $\Delta_+^\nc=\Delta_+\cap \Delta^\nc$. Then 
  $$\mf k_- =\bigoplus_{\alpha \in \Delta_+^{\mathrm{c}}}\mf g_\alpha,\, \mf p_-=\bigoplus_{\beta \in \Delta_+^{\mathrm{nc}}}\mf g_\beta.$$
\end{itemize}

With the above notations, we see that $\mf n_-$ or equivalently $\Delta_+$ determines the complex structure on $D$ via the integrable subbundle of the complexified differential tangent bundle,
$$\mathrm{T}^{1,0}D=G_\R\times_T \mf n_-\subset \mathrm{T}^{\C}D=G_\R\times_T \mf g/\mf h,$$
where the inclusion $\subset$ is given by
\begin{equation}\label{inclusion of tangent}
  \mf n_- \cong (\mf n_- \oplus \mf h)/\mf h \subset \mf g/\mf h.
\end{equation}
In the following, all the inclusions of integrable subbundles of the complex differential tangent bundles are given as \eqref{inclusion of tangent}.

Let $D=G_\R/T$ be a non-classical flag domain with $G_\R$ of Hermitian type, i.e. there exists a $G_\R$-invariant complex structure on the symmetric space $G_\R/K$. The $G_\R$-invariant complex structure on $G_\R/K$ is given by an abelian subspace $$\mf p_-'\subset \mf p,\, \mf p_-' \oplus \bar{\mf p_-'} = \mf p$$ such that
$$\mathrm{T}^{1,0}G_\R/K =G_\R\times_K \mf p'_-\subset \mathrm{T}^{\C}G_\R/K=G_\R\times_K \mf g/\mf k.$$

The following theorems are proved in \cite{LiuShen24}.

\begin{theorem}[Theorem 5.8 of \cite{LiuShen24}]\label{new complex D}
Let $D=G_\R/T$ be a non-classical flag domain, where $G_\R$ is of Hermitian type with an abelian subspace $\mf p_-'\subseteq \mf p$ such that $\mf p_-' \oplus \bar{\mf p_-'} = \mf p$. Then the flag domain $D$ admits a new $G_\R$-invariant complex structure, given by
\begin{equation}\label{new complex bundle}
  \T^{1,0} D' = G_\R\times_T \left( \mf k_-\oplus \mf p_-'\right)\subseteq \T^\C D
\end{equation}
where $D'$ denotes the differentiable manifold $D$ equipped with the new complex structure. Moreover, $D'\subset\check{D}$ is a classical flag domain with a holomorphic projection map
$$p':\, D'\to G_\R/K.$$
\end{theorem}

And furthermore we have 
\begin{theorem}[Theorem 6.3 of \cite{LiuShen24}]\label{Thm 6.3}
Let the assumption be as in Theorem \ref{new complex D} and let $\Gamma\subseteq G_\R$ be a co-compact and torsion-free discrete subgroup. 
Then we have two compact complex manifolds $X=\Gamma \backslash D$ and $X'=\Gamma \backslash D'$ which are diffeomorphic to each other, such that  $X$ is not in Fujiki class $\mathcal C$ while $X'$ is a projective manifold.
\end{theorem}

In fact, in \cite{LiuShen24}, we proved that $X=\Gamma \backslash D$ admits no  $\partial\bar{\partial}$-structure compactible with the complex structure induced from $D$, by solving a conjecture of Green--Griffiths--Kerr about the vanishing theorems of the zeroth automorphic cohomology of any non-trivial locally homogenous vector bundles on $X$. 

In \cite{CT}, Carlson and Toledo proved that there exists no K\"ahler metrics on $X$. 
In \cite{GRT}, Griffiths, Robles and Toledo proved that the quotient $X=\Gamma \backslash D$, not necessarily compact or smooth, is not algebraic. Hence Theorem \ref{Thm 6.3} can be considered as a generalization of their results.

On the other hand, it is proved in \cite{GS} that the compact quotient $X'$ of the classical flag domain $D'$ is a projective manifold.
Thus, it is of interest to endow the automorphic cohomology groups on $X$ with arithmetic structures inherited from those on $X'$ via the Penrose transformation discussed below.

Let us denote $\mf n_-'=\mf k_-\oplus \mf p_-'$ and the set of positive roots corresponding to the complex structure of $D'$ by $\Delta_+'$, i.e. 
$$\Delta_+'=\{\alpha\in \Delta:\, e_\alpha \in \mf n_-'\}.$$

Since $D$ is non-classical and $D'$ is classical, the complex structures given by $\mf n_-$ and $\mf n'_-$ are different.
Also note that $\mf n_-$ and $\mf n_-'$ only differ in $\mf p_-$ and $\mf p_-'$ and $$\mf p_- \oplus {\mf p_+} = \mf p_-' \oplus {\mf p_+'} =\mf p.$$ 
Hence if we let $\mf p_-^2=\mf p_-\cap \mf p'_-$ and $\mf p_-^1\subset \mf p_-$ such that $\mf p_-=\mf p_-^1\oplus \mf p_-^2$, then we have 
$$\mf p'_-=\mf p^1_+\oplus \mf p_-^2.$$
In this paper, we write $\mf s_+$ with subscript "$+$" for the complex conjugate of the subspace $\mf s_-$ with subscript "$-$" and vice versa.

In terms of roots, we have the union
$$\Delta_+^\nc = \Delta_+^{\nca}\cup \Delta_+^{\ncb}$$
where 
$$\Delta_+^{\nc,i}=\{\beta:\,e_\beta\in \mf p_-^{i}\},\, i=1,2.$$
For $\Delta_+'$, we have the union
$$\Delta_+'={\Delta'}_+^\c\cup {\Delta'}_+^\nc,\, \text{where }{\Delta'}_+^\c={\Delta}_+^\c,\, {\Delta'}_+^\nc = \left(-\Delta_+^{\nca}\right)\cup \Delta_+^{\ncb}.$$

\begin{lemma}\label{Delta' positive}
For any $\beta,\, \beta'\in {\Delta'}_+^\nc = \left(-\Delta_+^{\nca}\right)\cup \Delta_+^{\ncb}$, we have that $(\beta,\beta')\ge 0$.
\end{lemma}
\begin{proof}
Otherwise we have that $(\beta,\beta')< 0$, then $\beta+\beta' $ is a weight in $\Delta^\c$, which implies that 
$$[e_{\beta},e_{\beta'}]=N_{\beta,\beta'}e_{\beta+\beta'} \neq 0$$
for $e_\beta,,e_{\beta'} \in \mf p'_-$, which contradicts to that $\mf p'_-$ is abelian.
\end{proof}

Before closing this section, we introduce the following lemma for the proof of Lemma \ref{structure constant} below.

\begin{lemma}\label{relations of the Lie brackets}
Let the notations be as above. Then we have the following relations of the Lie brackets of the subspaces
\begin{eqnarray}
  &&[\mf k_-,\mf k_-] \subset \mf k_-; \label{kkk}\\
   &&[\mf k_-,\mf p_-^i] \subset \mf p_-^i,\, i=1,2; \label{kpipi}\\
  &&[\mf p_-^1,\mf p_-^2] \subset [\mf p_-,\mf p_-]\subset \mf k_-; \label{p1p2}\\
  && [\mf p_-^1,\mf p_+^2]=0 ,\,[\mf p_-^i,\mf p_-^i] =0,\, i=1,2.\label{pipi}
\end{eqnarray}
The above relations also hold if we reverse the signs.
\end{lemma}
\begin{proof}
\eqref{kkk} and \ref{p1p2} follow directly from the properties of Cartan decomposition in \ref{Cartan relation} and that $\mf n_-$ is a subalgebra.

\eqref{pipi} follows from that $$[\mf p_-^1,\mf p_+^2],\,[\mf p_-^1,\mf p_-^1]\subset [\mf p'_+,\mf p'_+]=0$$ and that
$$[\mf p_-^2,\mf p_-^2]\subset [\mf p'_-,\mf p'_-]=0.$$ 

Note that $\mf p'_-$ is preserved by $\mf k_\pm$ while $\mf p_-$ is only preserved by $\mf k_-$. Hence $$[\mf k_-,\mf p_-^1]\subset [\mf k_-,\mf p_-]\subset \mf p_- ,\, \, [\mf k_-,\mf p_-^1]\subset [\mf k_-,\mf p'_+]\subset \mf p'_+$$
which implies that 
$$[\mf k_-,\mf p_-^1]\subset \mf p_-\cap  \mf p'_+ =\mf p_-^1.$$
This proves \eqref{kpipi} with $i=1$. A similar argument implies \eqref{kpipi} with $i=2$.
\end{proof}
\begin{remark}\label{minimal p1}
 \textnormal{Lemma \ref{relations of the Lie brackets} has the counterpart relations of the corresponding subsets of roots. For example, \eqref{pipi} implies that for any $\beta,\, \beta' \in \Delta_+^{\nc ,i}$, the sum $\beta+\beta'$ is not a root, where $i=1,2$.}
\end{remark}


\section{Correspondence spaces}\label{Corres}
In this section, we introduce the correspondence space $\mathscr W$ for the flag domains $D$ and $D'$. We relate the cohomology groups of the homogenous line bundles on $D$ and $D'$ to the de Rham cohomology groups on $\mathscr W$, which is the foundation of the Penrose transformation. We also introduce the incidence variety $\mathscr I$ and prove the vanishing theorem on certain open subsets $\mathscr I^o$ of $\mathscr I$, which is crucial to the proof of the injectivity of the Penrose transformation.

Let $$D=G_\R/T\subset \check D =G_\C/B$$ be a non-classical flag domain with $G_\R$ of Hermitian type. Then we have a natural projection map
\begin{equation}\label{DtoGK 2}
  p:\, D=G_{\mathbb R}/V\to  G_{\mathbb R}/K.
\end{equation}
Let $o\in D$ be a base point and $$\bar{o}=p(o)\in G_\R/K.$$
The fiber 
\begin{equation}\label{base cycle}
  Z_o\triangleq p^{-1}(\bar o) \simeq K/V \simeq K_\C/(K_\C\cap B)
\end{equation}
is a flag manifold which is an analytic subvariety of $D$. We call $Z_o$ the base cycle for $D$.

We define the cycle space $\mathscr U$ by
\begin{equation}\label{defn of U}
  \mathscr U =\text{topological component of }Z_o \text{ in }\{gZ_o:\, g\in G_\C, gZ_o\subset D\}.
\end{equation}

Note that for the classical flag domain $D'$, $Z_o$ is also an analytic subvariety, while the cycle space $\mathscr U'$ for $D'$ is trivial in the sense that any $g\in G_\C$ such that $gZ_o\subset D'$ is an analytic subvariety if and only if $g\in G_\R$.

We collect some properties of the cycle space $\mathscr U$ as introduced in Lecture 6 of \cite{GGK}.

\begin{theorem}\label{cycle basics}
(1) The cycle space $\mathscr U$ is Stein.

(2) If $D$ is non-classical, then there is an open embedding $$\mathscr U \subset G_\C/K_\C\triangleq \check{\mathscr U}$$ 
where $K_\C$ is the closed subgroup of $G_\C$ corresponding to $\mf k \subset \mf g$.

(3) If $D$ is non-classical with $G_\R$ of Hermitian type, then there is a biholomorphism
$$\mathscr U \cong \mathbb B \times \bar{\mathbb B}$$
where $\mathbb B$ denotes the Hermitian symmetric space $G_\R/K$ and $\bar{\mathbb B}$ denotes the complex conjugate of $\mathbb B$.
\end{theorem}

Please see \cite{FHW} for the proofs of (1) and (2), and see \cite{BHH} or \cite{FHW} for the proof of (3).

We define some closed subgroups of the complex Lie group $G_\C$ as follows:
\begin{eqnarray*}
  T_\C &=&\text{complexification of $T$ with Lie algebra } \mf h; \\
  B_K &=& K_\C\cap B \text{ with Lie algebra } \mf b_K=\mf h\oplus \sum_{\alpha\in \Delta_+^\c}\mf g_{-\alpha},
\end{eqnarray*}
and define the corresponding complex homogenous spaces by
\begin{eqnarray*}
  \check{\mathscr W} &=& G_\C/T_\C;\\
  \check{\mathscr I} &=& G_\C /B_K.
\end{eqnarray*}
Since $T_\C \subset B_K\subset  B (\text{or } B')$ and $B_K\subset K_\C$, we have the following commutative diagram of projection maps of complex homogenous spaces
\begin{equation}\label{comm1}
\xymatrix{
&\check{\mathscr W}\ar[d]\ar[ldd]\ar[rdd]&\\
&\check{\mathscr I}\ar[ld]\ar[rd]&\\
\check{D}& &\check{\mathscr U}.
}
\end{equation}
Recall that the open embedding $D\subset \check D$ induces the open embedding $\mathscr U\subset \check{\mathscr U}$ in \eqref{defn of U}. We can also define the open subsets $\mathscr W \subset \check{\mathscr W}$ and $\mathscr I \subset \check{\mathscr I}$, such that each square box of the following diagram is Cartesian,
\begin{equation}\label{defn of WIJ}
\xymatrix{
\mathscr W\ar@{^{(}->}[r]\ar[d]&\check{\mathscr W}\ar[d]\\
\mathscr I\ar@{^{(}->}[r]\ar[d]&\check{\mathscr I}\ar[d]\\
 D\ar@{^{(}->}[r]&\check{\mathscr D}.
}
\end{equation}
Then the projection maps in \eqref{comm1} restrict to the following fibrations
\begin{equation}\label{comm2}
\xymatrix{
&{\mathscr W}\ar[d]^-{\pi_\mathscr{I}}\ar[ldd]_-{\pi}\ar[rdd]&\\
&{\mathscr I}\ar[ld]^-{\pi_D}\ar[rd]_-{\pi_\mathscr{U}}&\\
{D}& &{\mathscr U}
}
\end{equation}
We adopt the notations for fibrations from \cite{GGK} to facilitate reference, allowing readers to consult further details as needed.

\begin{definition}
The space $\check{\mathscr W}$ is referred to as the enhanced flag variety,  with $\mathscr W$ the correspondence space, and $\mathscr I$ is called the incidence variety.
\end{definition}

For $u\in \mathscr U$, we denote the corresponding analytic subvariety by $Z_u\subset D$. The incidence variety $\mathscr I$ can also be defined by
$$\mathscr I= \{(x,u)\in D\times \mathscr U:\, x\in Z_u\}.$$
Then the projection maps $\pi_D$ and $\pi_{\mathscr U}$ are the restrictions to $\mathscr I$ of the projections map $\mathrm{pr}_{1}:\,D\times \mathscr U \to D$ and $\mathrm{pr}_{2}:\,D\times \mathscr U \to  \mathscr U$ respectively.

\begin{proposition}
The space $\check{\mathscr W}$ is an affine algebraic variety and hence a Stein manifold; The open subset $\mathscr W$ of $\check{\mathscr W}$ is also Stein.
\end{proposition}
\begin{proof}
Since all the Lie groups considered in this paper are linear algebraic groups, the complex Lie groups $G_\C$ and $T_\C$, which are the complexificaions of the compact real Lie groups, are reductive.
Therefore the quotient $\check{\mathscr W} = G_\C/T_\C$ is an affine algebraic variety, and hence Stein. 

The open subset ${\mathscr W}\subset \check{\mathscr W}$ can also be defined as the fiber product $\mathscr U\times_{\check{\mathscr U}} \check{\mathscr{W}}$, which is a closed subset of $\mathscr U\times \check{\mathscr{W}}$. From Theorem \ref{cycle basics}, the product $\mathscr U\times \check{\mathscr{W}}$ of Stein manifolds is also Stein, and so is its closed submanifold ${\mathscr W}$.
\end{proof}

We also collect some basic facts about the fibers of the fibrations in \eqref{comm2}. Please see Lecture 7 in \cite{GGK} for details.

\begin{proposition}
(1) The fibers of $\pi_D:\, \mathscr{U} \to D$ and $\pi_\mathscr{I}:\, \mathscr{W}\to \mathscr{I}$ are isomorphic to the typical ones (i.e. fiber of the base point) $B/B_K\cong \mf p_+$ and $B_K/H \cong \mf k_+$ respectively which are contractible affine algebraic varieties.

(2) The fibers of $\pi_\mathscr{U}:\, \mathscr{I}\to \mathscr{U}$ are isomorphic to the typical one $Z_o\simeq K_\C/B_K$ which is a projective variety;

(3) The fibers of $\pi_\mathscr{I}:\, \mathscr{W}\to \mathscr{I}$ are isomorphic to the typical one $K_\C/T_\C$ which is an affine algebraic variety. 
\end{proposition}

In order to relate the cohomology on $D$ to that on $\mathscr W$, we recall the theorem of EGW (Eastwood-Gindikin-Wong) in \cite{EGW} following Lecture 7 of \cite{GGK}.

Let $ M, N $ be complex manifolds and $ \pi : M \to N $ a holomorphic submersion. 
For $ F \to N $ a holomorphic vector bundle, we let
\begin{itemize}
    \item $ \pi^{-1}F $ be the pullback to $ M $ of $F$ as a sheaf of holomorphic sections of $F$;
    \item $ \pi^*F =\pi^{-1}F \otimes_{\pi^{-1}\mathcal O_N}\mathcal O_M$ .
\end{itemize}
Then $ \pi^{-1}F \subset \pi^*F $ is identified with the sections of $ \pi^*F $ that are constant along the fibers of $ M \to N $.

We define $\Omega^1_\pi= \Omega^1_M/\pi^*\Omega^1_N$ and $ \Omega^q_\pi=\wedge^q \Omega^1_\pi$ to be the sheaf over $ M $ of relative holomorphic $ q $-forms, with the relative differential induced from $d$,
$$
d_\pi : \Omega^q_\pi \to \Omega^{q+1}_\pi.
$$
Then we have the short exact sequence 
$$
0 \to \pi^*\Omega^1_N \to \Omega^1_M \to \Omega^1_\pi \to 0
$$
of sheaves of holomorphic forms.

Tensoring with $\pi^{*}F$, we have $\Omega^q_\pi(F)=\Omega^q_\pi\otimes_{\mathcal O_M}\pi^{*}F$ and the relative differential with values in $F$,
$$
d_\pi : \Omega^q_\pi(F) \to \Omega^{q+1}_\pi(F).
$$
This gives the resolution of $\pi^{-1}F$
\begin{equation}\label{resl of piF}
  0 \to \pi^{-1}F \to \pi^{*}F=\Omega^0_\pi(F) \xrightarrow{d_\pi} \Omega^1_\pi(F) \xrightarrow{d_\pi} \Omega^2_\pi(F) \to \cdots.
\end{equation}

Let $\mathbb H^*(M,\Omega^\bullet_\pi(F))$ denote the hypercohomology of the complex $\Omega^\bullet_\pi(F)$ of sheaves. Then \eqref{resl of piF} implies that
$$H^*(M, \pi^{-1}F  )=\mathbb H^*(M,\Omega^\bullet_\pi(F)). $$

Let 
$$H^*_{DR}(M, \Omega^\bullet_\pi(F)) =H^*\left( \Gamma(M, \Omega^\bullet_\pi(F)); d_\pi \right)$$
be the de Rham cohomology groups, defined by the cohomology groups of the complex $$\left( \Gamma(M, \Omega^\bullet_\pi(F)); d_\pi \right)$$ of global holomorphic sections of $ \Omega^\bullet_\pi(F)$.


\begin{theorem}[EGW]\label{EGW}
  Assume that $ M $ is Stein and the fibers of $ M \to N $ are contractible. Then
\[
H^*(N, F) \cong H^*_{DR} \left(M, \Omega^\bullet_\pi(F)\right).
\]
\end{theorem}

Applying the theorem of EGW to the holomorphic submersion $\pi:\, \mathscr W\to D$, we have that
\begin{equation}\label{DtoW}
  H^*(D,L_\mu)=H^*_{DR}(\mathscr W,\Omega_\pi^\bullet(L_\mu)).
\end{equation}

Next we relate the de Rham cohomology groups on $\mathscr W$ to that on $\mathscr I$. 

From the commutative diagram of holomorphic submersions
$$\xymatrix{
&\mathscr W\ar[d]^-{\pi_\mathscr{I}}\ar[dl]^-{\pi}\\
D &\mathscr I\ar[l]^-{\pi_D},
}$$
we have the short exact sequence of sheaves on $\mathscr W$,
$$0\to \pi_\mathscr{I}^*\Omega_{\pi_D}^1\to \Omega_\pi^1\to  \Omega_{\pi_\mathscr{I}}^1 \to 0,$$
which defines a filtration $F$ on $\Omega_\pi^\bullet$ by
$$F^p\Omega_\pi^\bullet=\mathrm{Im}\{\pi_\mathscr{I}^*\Omega_{\pi_D}^p\otimes \Omega_\pi^{\bullet-p}\to \Omega_\pi^\bullet\}.$$

The spectral sequence for $F$ is
\begin{eqnarray*}
  &&E_0^{p,q}=\Gamma(\mathscr W, \mathrm{Gr}_F^p\Omega_\pi^{p+q}(L_\mu))= \Gamma(\mathscr W, \pi_\mathscr{I}^*\Omega_{\pi_D}^p\otimes \Omega_{\pi_\mathscr{I}}^{q}(L_\mu))\\
  &&d_0=d_{\pi_\mathscr{I}}:\,E_0^{p,q}\to E_0^{p,q+1}\\
  && E_\infty^{p,q}= \mathrm{Gr}_F^p H^{p+q}_{DR}(\mathscr W, \Omega_\pi^\bullet(L_\mu)).
\end{eqnarray*}
Then 
$$E_1^{p,q}=H^q_{DR}(\mathscr W,\pi_\mathscr{I}^*\Omega_{\pi_D}^p\otimes \Omega_{\pi_\mathscr{I}}^{\bullet}(L_\mu))_{d_{\pi_\mathscr{I}}}.$$
Here we write $H^*_{DR}(\cdots)_{d_{\pi_\mathscr{I}}}$ to emphasize the differential map.
Again apply the theorem of EGW to the holomorphic submersion $\pi_\mathscr{I}:\, \mathscr W\to \mathscr I$ with sheaf $\Omega_{\pi_D}^p(L_\mu)$ to have that
$$ E_1^{p,q}=H^q_{DR}(\mathscr W,\pi_\mathscr{I}^*\Omega_{\pi_D}^p\otimes \Omega_{\pi_\mathscr{I}}^{\bullet}(L_\mu))_{d_{\pi_\mathscr{I}}}=H^q(\mathscr I,\Omega_{\pi_D}^p(L_\mu)).$$
Hence
$$E_1^{p,0}=\Gamma(\mathscr I,\Omega_{\pi_D}^p(L_\mu)),\, d_1=d_{\pi_D}:\, E_1^{p,0}\to E_1^{p+1,0},$$
and 
$$E_2^{p,0}=H^p_{DR}(\mathscr I,\Omega_{\pi_D}^\bullet(L_\mu)).$$

The quotient maps 
$$E_2^{p,0}\to E_3^{p,0}\to \cdots \to E_\infty^{p,0}= \mathrm{Gr}_F^p H^{p}_{DR}(\mathscr W, \Omega_\pi^\bullet(L_\mu))$$
induce the surjective morphism $E_2^{p,0}\to E_\infty^{p,0}$, which is
$$\pi_\mathscr{I}^*:\, H^p_{DR}(\mathscr I,\Omega_{\pi_D}^\bullet(L_\mu)) \to H^{p}_{DR}(\mathscr W, \pi_\mathscr{I}^*\Omega_{\pi_D}^\bullet(L_\mu)).$$


In summary, we have proved the following.

\begin{proposition}\label{I to W surj}
The morphism
$$\pi_\mathscr{I}^*:\, H^*_{DR}(\mathscr I,\Omega_{\pi_D}^\bullet(L_\mu)) \to H^{*}_{DR}(\mathscr W, \pi_\mathscr{I}^*\Omega_{\pi_D}^\bullet(L_\mu))$$
induced by $\pi_\mathscr{I}:\, \mathscr W\to \mathscr{I}$ is surjective.
\end{proposition}

As the main result of this section, we prove a vanishing theorem on the space $\mathscr I$, which is key to the proof of the injectivity of the Penrose transformation.

Denote $$\rho_\c=\frac{1}{2}\sum_{\alpha\in \Delta_+^\c} \alpha.$$
Then we have the following vanishing theorem on $\mathscr I$ for the pull-backs of the homogenous line bundles on $D$.

\begin{theorem}\label{vanshing on I prop}
Let $L_\lambda$ be a homogenous line bundle on $D$ with weight $\lambda$.
If there exists some $\alpha \in \Delta_+^\c$ such that $(\lambda,\alpha)<0$, or equivalently $(\lambda+\rho_\c,\alpha)\le 0$,  then 
\begin{equation}\label{vanishing on I}
  \Gamma(\mathscr I, \pi_D^*L_\lambda)=0.
\end{equation}
\end{theorem}
\begin{proof}
\noindent \textbf{Step 1.} For any $G_\R$-translation of the base cycle $Z_o$ given in \eqref{base cycle}, we have that 
$$\Gamma(gZ_o, L_\lambda|_{gZ_o})=0, \, \forall\, g\in G_\R.$$

In fact, on the base cycle, we have
$$L_\lambda|_{Z_o}=K\times_T \C_\lambda$$
and hence the Bott--Borel--Weil theorem (c.f. \cite{Bott57}, \cite{GGK} or \cite{schmid97}) implies that
$$ \Gamma(Z_o, L_\lambda|_{Z_o})=H^0(Z_o, L_\lambda|_{Z_o})=0,$$
under the assumption of the theorem.

On the $G_\R$-translation $gZ_o$, the biholomorphic map $g:\, D\to D$ induces isomorphism 
$$g:\, L_\lambda|_{Z_o} \to L_\lambda|_{gZ_o}$$
which implies that 
$$\Gamma(gZ_o, L_\lambda|_{gZ_o})=\Gamma(Z_o, L_\lambda|_{Z_o})=0.$$

\noindent \textbf{Step 2.} There exists a nontrivial open subset $\mathscr U^o\subset \mathscr U$ containing the $G_\R$-translations of the base cycle $Z_o$ such that 
\begin{equation}\label{step2}
  \Gamma(Z_u, L_\lambda|_{Z_u})=0,\, \forall\, u\in \mathscr U^o.
\end{equation}

Let $u_0\in \mathscr U$ such that $Z_{u_0}=gZ_o$ for some $g\in G_\R$. 
We can choose a connected open neighborhood  $U_{u_0}$ of $u_0$ in $\mathscr U$, over which there exists an analytic family 
$$\xymatrix{
D \supset\mathcal Z \ar[r]^-{\pi}& U_{u_0}
}$$
which is $C^\infty$-trivial, such that $Z_u=\pi^{-1}(u)$ for any $u\in U_{u_0}$. Then the restriction
$$L_\lambda|_{\mathcal Z} \to \mathcal Z$$
gives an family of line bundles $$(L_\lambda|_{\mathcal Z})|_{Z_u}=L_\lambda|_{Z_u}$$ over $U_{u_0}$. 

By Theorem 6 in \cite{KS3}, we know that $\dim H^0(Z_u, L_\lambda|_{Z_u})$ is an upper semi-continuous function of $u\in U_{u_0}$, i.e. for any $u\in U_{u_0}$, one has
$$\dim H^0(Z_u, L_\lambda|_{Z_u})\le \dim H^0(Z_{u_0},L_\lambda|_{Z_{u_0}})=0$$
which implies that $\Gamma(Z_u, L_\lambda|_{Z_u})=0$ for $u\in U_{u_0}$.

Let $$\mathscr U^o=\bigcup_{{u}=[gZ_o],\, g\in G_\R}U_{u}.$$
Then $\mathscr U^o$ satisfies \eqref{step2}.

\noindent \textbf{Step 3.} Now we can prove the theorem.

Let 
$$\mathscr I^o= \{(x,u)\in D\times \mathscr U^o:\, x\in Z_u\}=\pi_{\mathscr U}^{-1}(\mathscr U^o).$$
Then $\mathscr I^o$ is an open subset of $\mathscr I$, and the restriction $$\pi_D|_{\mathscr I^o}:\, \mathscr I^o\to D$$ is still surjective, since for any $x\in D$ there exists $g\in G_\R$ such that $gZ_o$ passes through $x$.

For any fixed $(x,u)\in \mathscr I^o$ with $x\in Z_u$, the subset
$$\tilde{Z_u}=\{(y,u)\in \mathscr I^o:\, y\in Z_u\}$$
is a compact subvariety of $\mathscr I^o$ passing through $(x,u)$, which is isomorphic to $Z_u \subset D$ via $\pi_D$. Then
$$\Gamma(\tilde{Z_u}, \pi_D^*L_\lambda|_{\tilde{Z_u}})=\Gamma({Z_u}, L_\lambda|_{{Z_u}})=0$$
by Step 2.

Since $\{\tilde{Z_u}:\,u\in \mathscr U^o\}$ covers $\mathscr I^o$ and $\mathscr I^o$ is a nontrivial open subset of $\mathscr I$, we have the conclusion \eqref{vanishing on I} of the theorem.
\end{proof}

\begin{corollary}\label{vanshing on I coro}
Let the the assumption be as in Theorem \ref{vanshing on I prop}. Let $U$ be any small enough open subset of $\mathscr U$ containing the base point $u_o$, i.e. the point representing the base cycle $Z_o$. Then
$$\Gamma(\pi_{\mathscr U}^{-1}(U), \pi_D^*L_\lambda|_{\pi_{\mathscr U}^{-1}(U)})=0.$$
\end{corollary}

\begin{remark}
\textnormal{(1) In Corollary \ref{vanshing on I coro}, the open subset $ U \subset \mathscr{U} $ can be chosen sufficiently small such that the vanishing theorem remains valid. This is because $ \pi_{\mathscr{U}}^{-1}(U) $ always contains and is covered by the cycles $ \tilde{Z_u} $ for $ u \in U $.
However we do not have the vanishing theorem on the open subset $V\subset \mathscr I$, if $V$ dose not contain any $\tilde{Z_u}$ for $u\in \mathscr U$.}

\textnormal{(2) In \cite{GRT}, Griffiths, Robles and Toledo proved the cycle chain connectedness of the non-classical flag domains $D$, by proving that the incidence variety $\mathscr I$ can be connected by integral curves of a subbundle $S\simeq \mf k_-\oplus \mf p_+$ of $\mathrm{T}\mathscr I$. Now we can use the same method to show that $\mathscr I^o$ can be connected by integral curves of the restricted subbundle $S|_{\mathscr I^o}$, and conclude that any two points in $D$ can be connected by a chain of cycles in $\mathscr U^o$. Here $\mathscr U^o$ can be taken small enough containing the $G_\R$-translations of the base cycles.}

\textnormal{We will derive more geometric properties of non-classical flag domains from this observation in a forthcoming paper.}
\end{remark}


\section{Penrose transformation on flag domains}\label{section Pen}
In this section, we define the Penrose transformation on flag domains and establish its injectivity under specific conditions. Additionally, we present particular cases of weights that satisfy these conditions, so that the corresponding Penrose transformations are injective.

Let $D=G_\R/T$ be a non-classical flag domain with $G_\R$ of Hermitian type. 
Then Theorem \eqref{new complex D} implies that $D$ is diffeomorphic to a classical flag domain $D'$. The complex structures of $D$ and $D'$ are given respectively by
\begin{eqnarray*}
  \mathrm{T}^{1,0}_o D&\cong& \mf n_-=\mf k_-\oplus \mf p_-^1\oplus \mf p_-^2 \\
  \mathrm{T}^{1,0}_o D'&\cong& \mf n'_-=\mf k_-\oplus \mf p_+^1\oplus \mf p_-^2.
\end{eqnarray*}

We define the projection map
$$\pi_{D'}:\, \mathscr I\to D'$$
by sending $(x,u)\in \mathscr I$ to $x\in D'$, where $x\in D'\simeq D$ and $u=[Z_u]\in \mathscr U$. Then the fiber $\pi_{D'}^{-1}(x)$ is $\{u\in \mathscr U:\, x\in Z_u\}$. Note that there is only one $u\in \pi_{D'}^{-1}(x)$ such that $Z_u$ is the compact analytic subvariety of $D'$, which is the $G_\R$-translation of the base cycle $Z_o$. Nevertheless we still have that $\pi_{D'}$ remains a holomorphic map, as we have
$$\mathrm{T}^{1,0}_o D'\cong \mf n'_- \subset \mf k_-\oplus \mf p_-\oplus \mf p_+ \cong \mathrm{T}^{1,0}_o \mathscr I.$$

Therefore we have the following commutative diagram of holomorphic fibrations
\begin{equation}\label{Wto DD'}
  \xymatrix{
&\mathscr W\ar[ldd]_-{\pi}\ar[rdd]^-{\pi'}\ar[d]&\\
&\mathscr I\ar[ld]^-{\pi_{D}}\ar[rd]_-{\pi_{D'}}\\
D& &D'.}
\end{equation}
Since the fibers of $\pi'$ are isomorphic to those of $\pi$, we can apply Theorem \ref{EGW} of EGW to the holomorphic submersion $\pi':\, \mathscr W\to D'$ to have that
\begin{equation}\label{D'toW}
  H^*(D',L_{\mu'})=H^*_{DR}(\mathscr W,\Omega_{\pi'}^\bullet(L_{\mu'})).
\end{equation}

The Penrose transformation relates the cohomology $H^*(D',L_{\mu'})$ to $H^*(D,L_{\mu})$ via the de Rham cohomology groups of the global sections on $\mathscr W$ of relative differential forms with values in $L_{\mu'}$ and $L_\mu$. 

Before giving the precise definition, we first introduce some basic lemmas from Lie algebra and geometry of homogenous manifolds. 

In terms of root system, the complex structures on $D$ and $D'$ are given respectively by the set of positive roots
\begin{eqnarray*}
  \Delta_+&=&\Delta_+^\c\cup \Delta_+^{\nca}\cup \Delta_+^{\ncb} \\
  \Delta_+'&=&\Delta_+^\c\cup \left(-\Delta_+^{\nca}\right)\cup \Delta_+^{\ncb}.
\end{eqnarray*}
Let
\begin{eqnarray*}
  \rho &=& \frac{1}{2}\sum_{\alpha\in \Delta_+} \alpha=\frac{1}{2}\sum_{\alpha\in \Delta_+^\c} \alpha+\frac{1}{2}\sum_{\alpha\in \Delta_+^\nca} \alpha+\frac{1}{2}\sum_{\alpha\in \Delta_+^\ncb} \alpha\triangleq \rho_\c+\rho_\nca+\rho_\ncb\\
  \rho' &=& \frac{1}{2}\sum_{\alpha\in \Delta_+'} \alpha= \rho_\c-\rho_\nca+\rho_\ncb \triangleq \rho_\c+\rho'_\nc.
\end{eqnarray*}
For a weight $\lambda$ we define
\begin{eqnarray*}
  q(\lambda)&=&\#\{\alpha\in {\Delta}_+^\c:\, (\lambda,\alpha)<0\}+\#\{\alpha\in {\Delta}_+^\nc:\, (\lambda,\alpha)>0\} \\
  q'(\lambda)&=&\#\{\alpha\in {\Delta'}_+^\c:\, (\lambda,\alpha)<0\}+\#\{\alpha\in {\Delta'}_+^\nc:\, (\lambda,\alpha)>0\}.
\end{eqnarray*}

From now on, we fix the positive integer 
$$q=\#\Delta_+^\nca$$
which is the number of non-compact roots whose signs are changed when we transform the complex structure on $D$ to that on $D'$.

For convenience we write $\alpha$, $\beta$ and $\gamma$ for the roots in $\Delta_+^\c$, $\Delta_+^\nca$ and $\Delta_+^\ncb$ respectively. Then
\begin{eqnarray*}
  \Delta_+^\c &=& \{\alpha_1,\cdots, \alpha_d\}, \\
  \Delta_+^\nca &=& \{\beta_1,\cdots, \beta_q\}, \\
  \Delta_+^\ncb &=& \{\gamma_1,\cdots, \gamma_p\}
\end{eqnarray*}
where $d=\dim_\C Z_o=\#\Delta_+^\c$ and $p=\dim_\C D -d-q=\#\Delta_+^\ncb$.

Let $\omega^\delta$ be the dual to $e_\delta$ for $\delta\in \Delta$. Then $\omega^\delta$ can be considered as a left invariant differential form on $G_\C$, which
satisfies that
$$\omega^\delta(g.X)=-\delta(X)\omega^\delta(g),\,\forall\, X\in \mathfrak h,\, \forall\, g\in G_\C.$$
Here $g.X$ denotes the right infinitesimal action of $\mathfrak h$ on $G_\C$. 
Hence, $\omega^\delta$ descends to a holomorphic left-invariant differential $1$-form with values in $L_{\delta}$ on $\check{\mathscr{W}} = G_\C / T_\C$ and then restricts to the open subset $\mathscr{W}$, i.e.
$$\omega^\delta\in \Gamma(\mathscr{W}, \Omega_{\mathscr{W}}^1(L_{\delta})).$$

\begin{lemma}\label{structure constant}
Let $C_{\delta\delta'}^{\delta''}$, $\delta,\delta',\delta''\in \Delta_+$, be the structure constants of the algebra $\mf n_+=\mf k_+\oplus \mf p_+^1\oplus \mf p_+^2$ such that
$$[e_{-\delta},e_{-\delta'}]=\sum_{\delta''\in \Delta_+} C_{\delta\delta'}^{\delta''}e_{-\delta''},\, C_{\delta\delta'}^{\delta''}=-C_{\delta'\delta}^{\delta''}.$$
Then 
\begin{eqnarray}
&& C_{\alpha_i\alpha_{i'}}^{\beta_j}=0,\,  C_{\alpha_i\alpha_{i'}}^{\gamma_k}=0; \label{cc}\\
&& C_{\alpha_i\beta_j}^{\alpha_{i'}}=0,\,  C_{\alpha_i\beta_j}^{\gamma_k}=0;\,  C_{\alpha_i\gamma_k}^{\alpha_{i'}}=0,\, C_{\alpha_i\gamma_k}^{\beta_{j}}=0; \label{cnc}\\
&& C_{\beta_j\gamma_k}^{\beta_{j'}}=0,\,  C_{\beta_j\gamma_k}^{\gamma_{k'}}=0;\label{ncnc}\\
&& C_{\beta_j\beta_{j'}}^{\delta}=0,\,  C_{\gamma_k\gamma_{k'}}^{\delta}=0,\, \forall\, \delta\in \Delta_+,\label{ncinci}
\end{eqnarray}
for $1\le i,i'\le d$, $1\le j,j' \le q$, $1\le k,k' \le r$.

\end{lemma}
\begin{proof}
Equation \eqref{cc}, \eqref{cnc},\eqref{ncnc},\eqref{ncinci} follow from \eqref{kkk}, \eqref{kpipi}, \eqref{p1p2}, \ref{pipi} in Lemma \ref{relations of the Lie brackets} respectively.
\end{proof}

In the following, we write $$\omega\equiv_\pi \omega'\, \, \mbox{for}\,  \,\omega,\, \omega'\in  \Gamma(\mathscr{W}, \Omega_{\pi}^\bullet(L_{\lambda}))$$ if $\omega-\omega'=0$ in $\Omega_{\pi}^\bullet(L_{\lambda})$.

\begin{lemma}\label{dpi omgega}
Let
\begin{eqnarray*}
   && \omega^{-\alpha_i}\in \Gamma(\mathscr{W}, \Omega_{\pi}^1(L_{-\alpha_i})),\, \alpha_i\in \Delta_+^\c,\, 1\le i\le d; \\
  && \omega^{-\beta_j}\in  \Gamma(\mathscr{W}, \Omega_{\pi}^1(L_{-\beta_j})),\, \alpha_j\in \Delta_+^\nca,\, 1\le j\le q; \\
  && \omega^{-\gamma_k}\in  \Gamma(\mathscr{W}, \Omega_{\pi}^1(L_{-\gamma_k})),\, \gamma_k\in \Delta_+^\ncb,\, 1\le k\le p.
\end{eqnarray*}
Then
\begin{eqnarray}
 && d_\pi \omega^{-\alpha_i}\equiv_\pi -\frac{1}{2}\sum_{i',i''}C_{\alpha_{i'}\alpha_{i''}}^{\alpha_i} \omega^{-\alpha_{i'}} \wedge  \omega^{-\alpha_{i''}}- \sum_{j,k}C_{\beta_{j}\gamma_{k}}^{\alpha_i} \omega^{-\beta_{j}} \wedge  \omega^{-\gamma_{k}};\label{dalpha} \\
 && d_\pi \omega^{-\beta_j}\equiv_\pi - \sum_{i,j'}C_{\alpha_{i}\beta_{j'}}^{\beta_j} \omega^{-\alpha_{i}} \wedge  \omega^{-\beta_{j'}};\label{dbeta}\\
 && d_\pi \omega^{-\gamma_{k}}\equiv_\pi - \sum_{i,k'}C_{\alpha_{i}\gamma_{k'}}^{\gamma_{k}} \omega^{-\alpha_{i}} \wedge  \omega^{-\gamma_{k'}}\label{dgamma}.
\end{eqnarray}
\end{lemma}
\begin{proof}
From (3.8) of \cite{schmid67} (see also \cite{GS} Page 260), we have that 
\begin{eqnarray}
  &&-2d_\pi \omega^{-\alpha_i} \equiv_\pi  \sum_{\delta,\delta'\in \Delta_+} C_{\delta\delta'}^{\alpha_i}\omega^{-\delta}\wedge \omega^{-\delta'}\nonumber\\
  &&= \sum_{i',i''}C_{\alpha_{i'}\alpha_{i''}}^{\alpha_i} \omega^{-\alpha_{i'}}\wedge \omega^{-\alpha_{i''}}+\sum_{i',j}C_{\alpha_{i'}\beta_j}^{\alpha_i} \omega^{-\alpha_{i'}}\wedge \omega^{-\beta_j}+\sum_{i',k}C_{\alpha_{i'}\gamma_k}^{\alpha_i} \omega^{-\alpha_{i'}}\wedge \omega^{-\gamma_k}\nonumber\\
   &&\,\,\,\,\,\,\,+\sum_{j,i'}C_{\beta_j\alpha_{i'}}^{\alpha_i} \omega^{-\beta_j}\wedge \omega^{-\alpha_{i'}}+\sum_{j,j'}C_{\beta_j\beta_{j'}}^{\alpha_i} \omega^{-\beta_j}\wedge \omega^{-\beta_{j'}}+\sum_{j,k}C_{\beta_j\gamma_k}^{\alpha_i} \omega^{-\beta_j}\wedge \omega^{-\gamma_k} \nonumber\\
   &&\,\,\,\,\,\,\,+\sum_{k,i'}C_{\gamma_k\alpha_{i'}}^{\alpha_i} \omega^{-\gamma_k}\wedge \omega^{-\alpha_{i'}}+\sum_{k,j}C_{\gamma_k\beta_j}^{\alpha_i} \omega^{-\gamma_k}\wedge \omega^{-\beta_{j}}+\sum_{k,k'}C_{\gamma_k\gamma_{k'}}^{\alpha_i}\omega^{-\gamma_k}\wedge \omega^{-\gamma_{k'}} .\nonumber
\end{eqnarray}
Then \eqref{dalpha} follows from Lemma \ref{structure constant} and that $C_{\beta_{j}\gamma_{k}}^{\alpha_i}=-C_{\gamma_{k}\beta_{j}}^{\alpha_i}$. 

Similarly we can deduce \eqref{dbeta} and \ref{dgamma} from Lemma \ref{structure constant} and the anti-symmetry of the structure constants.
\end{proof}

\begin{lemma}\label{omega nc1}
Let $$\omega^\nca=\omega^{-\beta_1}\wedge \cdots \wedge \omega^{-\beta_q}$$ be a left invariant differential form on $G_\C$. Then $\omega^\nca$ descends to a left invariant differential $q$-form with values in $L_{-2\rho_\nca}$ on $\mathscr W$ which is $d_\pi$-closed, i.e. 
$$\omega^\nca\in  \mathcal Z_{d_{\pi}}(\mathscr W, \Omega_\pi^q(L_{-2\rho_\nca})).$$
\end{lemma}
\begin{proof}
With the discussion before Lemma \ref{structure constant}, we only need to prove that $d_\pi(\omega^\nca)=0$.

From \eqref{dbeta}, we have that
$$d_\pi \omega^{-\beta_j}\equiv_\pi - \sum_{i,j'}C_{\alpha_{i}\beta_{j'}}^{\beta_j} \omega^{-\alpha_{i}} \wedge  \omega^{-\beta_{j'}},$$
and the structure constants
$$C_{\alpha_{i}\beta_{j'}}^{\beta_j} \neq 0\text{ only if } j\neq j'\text{ and $\alpha+\beta_{j'}=\beta_j$ }.$$
Therefore
\begin{eqnarray*}
  d_\pi(\omega^\nca) &=& \sum_{1\le j\le q}(-1)^{j-1} \omega^{-\beta_1}\wedge \cdots \wedge d_\pi \omega^{-\beta_j} \wedge\cdots \wedge \omega^{-\beta_q} \\
   &\equiv_\pi& \sum_{1\le j\le q}(-1)^{j} \omega^{-\beta_1}\wedge \cdots \wedge \left(\sum_{i,j'\neq j}C_{\alpha_{i}\beta_{j'}}^{\beta_j} \omega^{-\alpha_{i}} \wedge  \omega^{-\beta_{j'}}\right) \wedge\cdots \wedge \omega^{-\beta_q}\\
   &=&0.
\end{eqnarray*}
\end{proof}

\begin{lemdef}\label{Pen defn}
Let $\mu$ and $\mu'$ be two weights such that $\mu+\rho=\mu'+\rho'$.
Then the Penrose transformation $$\mathscr P:\, H^0(D',L_{\mu'})\to H^q(D,L_{\mu})$$ is defined by
\begin{equation}\label{PenD}
  \xymatrix{
  H^0(D',L_{\mu'}) \ar[d]^-{\cong}\ar[rr]^-{\mathscr P}&& H^q(D,L_{\mu})\ar[d]^-{\cong}\\
  H^0_{DR}(\mathscr W,\Omega_{\pi'}^\bullet(L_{\mu'}))\ar[rr]^-{\omega^\nca}&&H^q_{DR}(\mathscr W,\Omega_{\pi}^\bullet(L_{\mu})).
  }
\end{equation}
\end{lemdef}
\begin{proof}
We need to check the well-definedness of $\mathscr P$ given as above.

For any $\sigma\in H^0(D',L_{\mu'})$, let $$F_\sigma \in \mathcal Z_{d_{\pi'}}(\mathscr W,L_{\mu'})=H^0_{DR}(\mathscr W,\Omega_{\pi'}^\bullet(L_{\mu'}))$$ correspond to $\sigma$ via the left vertical isomorphism in \eqref{PenD}. 

For convenience, we will not distinguish the groups and their elements in the vertical isomorphisms in \eqref{PenD}.

Note that $\mu+\rho=\mu'+\rho'$ implies that $\mu=\mu'-2\rho_\nca$. Hence Lemma \ref{omega nc1} implies that $$F_\sigma \omega^\nca\in \Gamma(\mathscr W, \Omega_\pi^q(L_{\mu})).$$
We need to check that $d_\pi(F_\sigma \omega^\nca)=0$.

Since the complex structures of $D$ and $D'$ only differ in the directions corresponding to $\mf p_-^1$, we have that $$(d_\pi-d_{\pi'})F_\sigma \equiv_\pi\sum_{j}a_j \omega^{-\beta_j}.$$
Therefore from Lemma \ref{omega nc1} we have that
\begin{eqnarray*}
  d_\pi(F_\sigma \omega^\nca) &=& d_\pi(F_\sigma) \wedge\omega^\nca +F_\sigma d_\pi(\omega^\nca)\\
   &\equiv_\pi& (d_\pi-d_{\pi'})(F_\sigma) \wedge\omega^\nca\\
   &\equiv_\pi&\sum_{j}a_j \omega^{-\beta_j}\wedge\omega^\nca=0.
\end{eqnarray*}
Hence
$$F_\sigma \omega^\nca\in \mathcal Z_{d_\pi}(\mathscr W, \Omega_\pi^q(L_{\mu})),$$
and the Penrose transformation given by \eqref{PenD} is well-defined.
\end{proof}

\begin{theorem}\label{mainD}
Let $D=G_\R/T$ be a non-classical flag domain with $G_\R$ of Hermitian type. Let $D'$ be the classical flag domain which is diffeomorphic to $D$.

If $\mu$ and $\mu'$ are two weights with $\mu+\rho=\mu'+\rho'$, and satisfy that for any $\beta\in\Delta_+^\nca$ there exists an $\alpha_\beta\in \Delta_+^\c$ such that 
\begin{equation}\label{Pen inj conditions}
  (\alpha_\beta,\mu'-\beta)<0,
\end{equation}
then the Penrose transformation $$\mathscr P:\, H^0(D',L_{\mu'})\to H^q(D,L_{\mu})$$ given by \eqref{PenD} is injective.
\end{theorem}
\begin{proof}
Let $F_\sigma \in Z_{d_{\pi'}}(\mathscr W,L_{\mu'})$ such that $$\mathscr P([F_\sigma])=[F_\sigma \omega^\nca]=0$$ in $H^q_{DR}(\mathscr W,\Omega_{\pi}^\bullet(L_{\mu}))$.
Then there exists $\Psi\in \Gamma(\mathscr W,\Omega_{\pi}^{q-1}(L_{\mu}))$ such that 
$$F_\sigma \omega^\nca=d_\pi\Psi.$$
We need to show that $F_\sigma=0$.

Since $F_\sigma \omega^\nca$ has only term in $\wedge^q (\mf p_+^1)^*=\C\{\omega^\nca\}$, we see that $$[F_\sigma \omega^\nca]\in H^{q}_{DR}(\mathscr W, \pi_\mathscr{I}^*\Omega_{\pi_D}^\bullet(L_\mu)).$$
By Proposition \ref{I to W surj}, there exists $$[\Phi]\in H^q_{DR}(\mathscr I,\Omega_{\pi_D}^\bullet(L_\mu))$$ such that $\pi_{\mathscr I}^*[\Phi]=[F_\sigma \omega^\nca]$. Hence we can assume that $F_\sigma$ is constant along the fibers of $\pi_\mathscr{I}:\, \mathscr W\to \mathscr I$.

Let $\Psi=\Psi'+\Psi''$, where 
\begin{eqnarray*}
  \Psi'&=&\sum_{j}f_{j}\omega^{-\beta_1}\wedge \cdots \wedge \widehat{\omega^{-\beta_j}} \wedge\cdots \wedge \omega^{-\beta_q}; \\
   \Psi''&=&\sum_i g_i\omega^{-\alpha_{i}}\wedge \cdots +\sum_k h_k\omega^{-\gamma_{k}}\wedge \cdots.
\end{eqnarray*}
Then Lemma \ref{dpi omgega} implies that
\begin{eqnarray}
  F_\sigma \omega^\nca &\equiv_\pi& d_\pi \Psi'+d_\pi \Psi'' \nonumber\\
  &\equiv_\pi& \sum_j (-1)^{j-1}e_{-\beta_j}(f_j)\omega^\nca \label{dPsi}\\
   &&\pm\sum_{j\neq j'} f_{j}d_\pi\omega^{-\beta_{j'}}\wedge\omega^{-\beta_1}\wedge \cdots  \widehat{\omega^{-\beta_{j'}}} \cdots \widehat{\omega^{-\beta_j}}\wedge\cdots \nonumber \\ &&+\sum_{i,j}e_{-\alpha_i}(f_j)\omega^{-\alpha_i}\wedge\omega^{-\beta_1}\wedge \cdots \wedge \widehat{\omega^{-\beta_j}} \wedge\cdots \nonumber \\
  &&+\sum_{j,k}e_{-\gamma_k}(f_j)\omega^{-\gamma_k}\wedge\omega^{-\beta_1}\wedge \cdots \wedge \widehat{\omega^{-\beta_j}} \wedge\cdots + d_\pi \Psi''. \nonumber
\end{eqnarray}
From \eqref{dalpha} and \eqref{dgamma}, we see that $d_\pi \Psi''$ and
$$\pm\sum_{j\neq j'} f_{j}d_\pi\omega^{-\beta_{j'}}\wedge\omega^{-\beta_1}\wedge \cdots  \widehat{\omega^{-\beta_{j'}}} \cdots \widehat{\omega^{-\beta_j}}\wedge\cdots$$
have no terms in $\wedge^q (\mf p_+^1)^*=\C\{\omega^\nca\}$. Hence by comparing the types of both sides of \eqref{dPsi}, we have that
\begin{eqnarray}
 && F_\sigma \omega^\nca = \sum_j (-1)^{j-1}e_{-\beta_j}(f_j)\omega^\nca \triangleq \partial_{\mf p_+^1} \Psi',\label{dPsi type1}
 \end{eqnarray}
and 
\begin{eqnarray} 
 &&\pm\sum_{j\neq j'} f_{j}d_\pi\omega^{-\beta_{j'}}\wedge\omega^{-\beta_1}\wedge \cdots  \widehat{\omega^{-\beta_{j'}}} \cdots \widehat{\omega^{-\beta_j}}\wedge\cdots+ \nonumber\\
 && \sum_{i,j}e_{-\alpha_i}(f_j)\omega^{-\alpha_i}\wedge \cdots +\sum_{j,k}e_{-\gamma_k}(f_j)\omega^{-\gamma_k}\wedge\cdots + d_\pi \Psi''=0. \nonumber
\end{eqnarray}
Here we denote
$$\partial_{\mf p_+^1}(f\omega)=\sum_j e_{-\beta_j}(f)\omega^{-\beta_j}\wedge \omega$$
which is globally defined on $\mathscr W$.
%
%
%

Consider the fibrations in \eqref{comm2}
$$\xymatrix{
\mathscr W \ar@{^{(}->}[r] \ar[d]^-{\pi_{\mathscr I}} \ar@/_2pc/[dd]_-{\pi'}& {G_\C/T_\C} \ar[d]\\
\mathscr I\ar@{^{(}->}[r]\ar[d]^-{\pi_{\mathscr U}}& G_\C/B_K \ar[d]\\
\mathscr U=\mathbb B\times \bar{\mathbb B}\ar@{^{(}->}[r]& G_\C/K_\C
}$$
where $\mathbb B$ denotes the Hermitian symmetric domain $G_\R/K$, which is isomorphic to a boundede domain $B$ in $\mf p_-$ by the Harish-Chandra's embedding theorem, c.f. Lemma 7.11 in pages 390 -- 391 in \cite{Hel}.

Let $$U=U_1\times U_2 \subset \mathscr U$$ be an open subset of $\mathscr U$ with $U_1$ open subset of $\mathbb B$ and $U_2=\bar{U_1}\subset \bar{\mathbb B}$. We take $U$ small enough so that $$\pi_\mathscr{U}^{-1}(U)\cong K_\C/B_K\times U,\, {\pi'}^{-1}(U)=\pi_\mathscr{I}^{-1}\pi_\mathscr{U}^{-1}(U)\cong K_\C/T_\C\times U$$
and $\pi_\mathscr{U}^{-1}(U)\subset \mathscr I ^o$ is covered by the compact analytic subvarieties $\tilde{Z}_u$, $u\in U$, on which $\Gamma(\tilde{Z}_u, \pi_D^*L_{\mu_j})=0$ for line bundles $L_{\mu_j}$, $1\le j\le q$, which is to be defined below. Here $\mathscr I ^o$ and $\tilde{Z}_u$ are given as in the proof of Theorem \ref{vanshing on I prop}.

Then $\pi_{\mathscr I}$ is locally given by
$$\pi_{\mathscr I}|_{{\pi'}^{-1}(U)}=\pi_{B_K}\times \mathrm{id}_U:\,  K_\C/T_\C\times U\to K_\C/B_K\times U$$
where $$\pi_{B_K}:\, K_\C/T_\C \to K_\C/B_K$$ is the projection map with contractible fibers $B_K/T_\C\simeq \mf k_+$.

Note that all the homogenous vector bundles on the homogenous manifolds $X$ can be trivialized simultaneously under an open cover of $X$.
Let $\{V_\nu\}$ be a cover of the flag variety $Z_o\cong K_\C/B_K$ under which all the homogenous vector bundles on $K_\C/B_K$  are trivialized. Then, since $B_K/T_\C$ is contractible, $\{B_K/T_\C\times V_\nu\}$ is a cover of $K_\C/T_\C$ under which all the homogenous vector bundles on $K_\C/T_\C$ are trivialized.

Introduce coordinates 
\begin{eqnarray*}
&&  B_K/T_\C =\{w\in \mf k_+:\, w=(w_{1},\cdots, w_{d})\},\\
&&  V_\nu= \{x_\nu\in \mf k_-:\, x_\nu=(x_{\nu 1},\cdots, x_{\nu d})\}, \\
&&  U_1 = \{(y,z)\in B:\, y=(y_1,\cdots,y_q), z=(z_1,\cdots,z_r)\}\subset \mf p_-=\mf p_-^1\oplus \mf p_-^2
\end{eqnarray*}
 on the corresponding spaces, with  $B$ a bounded subset of $\C^{q+r}$. Then
$$\Psi'|_{B_K/T_\C\times V_\nu \times U}=\sum_{j}f_{j\nu}(w,x_\nu,y,\bar{y},z,\bar{z})\omega^{-\beta_1}\wedge \cdots \wedge \widehat{\omega^{-\beta_j}} \wedge\cdots \wedge \omega^{-\beta_q},$$
where $w\in B_K/T_\C$, $x_\nu \in V_\nu$ and $(y,\bar{y},z,\bar{z})\in U$, and 
$$e_{-\alpha_i}(f_{j\nu})=\partial_{w_i} (f_{j\nu}),\,e_{-\beta_{j'}}(f_{j\nu})=\partial_{\bar y_{j'}} (f_{j\nu}),\,e_{-\gamma_k}(f_{j\nu})=\partial_{\bar z_k} (f_{j\nu}).$$

Now we define $\Xi \in \Gamma({\pi'}^{-1}(U),\Omega_{\pi}^{q-1}(L_{\mu}))$ by
$$\Xi|_{B_K/T_\C\times V_\nu \times U}=\sum_{j}f_{j\nu}(0,x_\nu,y,\bar{y},z,\bar{z})\omega^{-\beta_1}\wedge \cdots \wedge \widehat{\omega^{-\beta_j}} \wedge\cdots \wedge \omega^{-\beta_q}.$$
Then \eqref{dPsi type1} implies that
\begin{eqnarray*}
  \partial_{\mf p_+^1} \Xi|_{B_K/T_\C\times V_\nu \times U}&= & \sum_j (-1)^{j-1}\partial_{\bar y_{j}}(f_{j\nu}(0,x_\nu,y,\bar{y},z,\bar{z})\omega^\nca \\
   &=& F_\sigma(0,x_\nu,y,\bar{y},z,\bar{z})\omega^\nca|_{B_K/T_\C\times V_\nu \times U}\\
  &=& F_\sigma(w,x_\nu,y,\bar{y},z,\bar{z}) \omega^\nca|_{B_K/T_\C\times V_\nu \times U}
\end{eqnarray*}
where the last equation follows from that $F_\sigma$ is constant along the fibers of $\pi_\mathscr I:\, \mathscr W\to \mathscr I$.
This implies that, when restricted to ${\pi'}^{-1}(U)$,
$$\partial_{\mf p_+^1}\Xi = F_\sigma\omega^\nca|_{{\pi'}^{-1}(U)}.$$

By construction, $\Xi$ is constant in $w\in B_K/T_\C$, i.e. constant along the fibers of $\pi_\mathscr I:\, \mathscr W\to \mathscr I$, and hence descends to $\pi_\mathscr{U}^{-1}(U)\subset \mathscr I$,
$$\Xi \in \bigoplus_{1\le j\le q}\Gamma(\pi_\mathscr{U}^{-1}(U), L_{\mu_j})$$
where $$\mu_j=\mu+\beta_1+\cdots +\widehat{\beta_j}+\cdots+\beta_q=\mu'-\beta_{j}.$$

By Borel-Weil-Bott theorem, condition \eqref{Pen inj conditions} implies that on the base cycle $Z_o$
$$\Gamma(Z_o, L_{\mu_j})=0,\, 1\le j\le q$$
and so does on its small deformations ${Z}_u$, $u\in U$.
Since $\pi_\mathscr{U}^{-1}(U)$ is covered by the compact analytic subvarieties $\tilde{Z}_u$, $u\in U$, which are isomorphic to $Z_u$ via the map $\pi_D$, we have that
$$\Gamma(\pi_\mathscr{U}^{-1}(U), L_{\mu_j})=0,\, 1\le j\le q.$$

Therefore $\Xi=0$ and $$F_\sigma \omega^\nca|_{{\pi'}^{-1}(U)}=\partial_{\mf p_+^1}\Xi=0$$
on the open subset ${{\pi'}^{-1}(U)}\subset \mathscr W$, which implies that $F_\sigma=0$ on $\mathscr W$. This proves the injectivity of the Penrose transformation.
\end{proof}

We need to consider the case that the Penrose transformation is non-trivial, i.e. $$H^0(D',L_{\mu'})\neq 0.$$ A necessary condition for this is given as follows.

\begin{theorem}\label{Pen nontrivial}
Let the notations be as Theorem \ref{mainD}. If $H^0(D',L_{\mu'})\neq 0$, then
\begin{equation}\label{Pen nontrivial conditions}
  (\mu',\alpha)\ge 0,\, \forall\, \alpha\in \Delta_+^\c.
\end{equation}
\end{theorem}
\begin{proof}
Assume on the contrary that there exists an $\alpha \in \Delta_+^\c$ such that 
\begin{equation}\label{D' nontrivial 1}
 (\mu',\alpha)<0.
\end{equation}

Note that
$$L_{\mu'}|_{Z_o}=K\times_T \C_{\mu'}$$
where $Z_o=K/T\subset D'$ is the base cycle through the base point.
Then from Borel-Weil-Bott theorem, we have that
$$H^0(Z_o, L_{\mu'}|_{Z_o})=H^0(Z_o, K\times_T \C_{\mu'})=0$$
under the assumption \eqref{D' nontrivial 1}.

For any $g\in G_\R$, $$g:\, L_{\mu'}|_{Z_o}\to L_{\mu'}|_{gZ_o}$$ induces an isomorphism of line bundles, and hence
$$H^0(gZ_o, L_{\mu'}|_{gZ_o})=0.$$
Since $D'$ is covered by the cycles $gZ_o$, $g\in G_\R$, we have that $H^0(D',L_{\mu'})= 0$, which contradicts to the assumption of the theorem. Hence \eqref{Pen nontrivial conditions} holds.
\end{proof}

From conditions \eqref{Pen inj conditions} and \eqref{Pen nontrivial conditions} in Theorem \ref{mainD} and Theorem \ref{Pen nontrivial}, it is reasonable to consider the weight $\mu'$ satisfying that
\begin{equation}\label{mu alpha =0}
  (\mu',\alpha)= 0,\, \forall\, \alpha\in \Delta_+^\c.
\end{equation}
Then the injectivity condition \eqref{Pen inj conditions} of the Penrose transformation is reduced to the conclusion of Lemma \ref{beta alpga >0} below. 

By the way, since $D$ is non-classical, the equations 
$$(\lambda, \alpha)= 0,\, \alpha\in \Delta_+^\c $$
define a linear complex subspace of $\mf h^*$ of codimension $\ge 1$. Hence the set $\Sigma$ of weights $\mu'$ satisfying \eqref{mu alpha =0} is non-empty and once $\mu'_0\in \Sigma$, then $k\mu'_0\in \Sigma$ for $k\in \mathbb Q$ such that $k\mu'_0$ is a weight.

\begin{theorem}\label{mainD example}
Let $D=G_\R/T$ be a non-classical flag domain with $G_\R$ of Hermitian type. Let $D'$ be the classical flag domain which is diffeomorphic to $D$.
If $\mu'$ is a weight such that
$$(\mu',\alpha)= 0,\, \forall\, \alpha\in \Delta_+^\c,$$
and $\mu=\mu'-2\rho_\nca$,
then the Penrose transformation $$\mathscr P:\, H^0(D',L_{\mu'})\to H^q(D,L_{\mu})$$ given by \eqref{PenD} is injective.
\end{theorem}
\begin{proof}
The theorem follows from Theorem \ref{mainD} and the following lemma.
\end{proof}

\begin{lemma}\label{beta alpga >0}
For any $\beta\in \Delta_+^\nca$, there exists an $\alpha\in \Delta_+^\c$ such that $(\beta,\alpha)>0$.
\end{lemma}
\begin{proof}
We first prove the lemma under the assumption that $G_\R$ is simple.

Let $$\Delta_+^{\nc,12}=\{\beta\in\Delta_+^\nca:\, \exists\, \gamma\in \Delta_+^\ncb,\text{ s.t. }\beta+\gamma\in\Delta_+^\c\}\subset \Delta_+^\nca.$$
We claim that for any $\beta\in \Delta_+^{\nc,12}$, there exists an $\alpha\in \Delta_+^\c$ such that $(\beta,\alpha)>0$.

Let $\gamma\in \Delta_+^\ncb$ such that $\beta+\gamma\in \Delta_+^\c$ is a root.
Then 
$$(\beta,\beta+\gamma)=(\beta,\beta)+(\beta,\gamma).$$
From Lemma \ref{Delta' positive}, we know that $(\beta,\gamma)\le 0$.

If $(\beta,\gamma)=0$, then $(\beta,\beta+\gamma)>0$, and we are done. Otherwise, $(\beta,\gamma)<0$, and $$\gamma+\beta,\cdots,\gamma+k\beta \text{ are roots, where } k=-\frac{2(\beta,\gamma)}{(\beta,\beta)}.$$
From Table 1 in Page 45 of \cite{Hum}, the possible cases are $k=1,2,3$.

If $k=1,3$, then $\gamma+k\beta\in \Delta_+^\c$ and $$(\beta, \gamma+k\beta)=(\beta,\gamma)+k(\beta,\beta)=\frac{k}{2}(\beta,\beta)>0.$$

The remaining case is $k=2$. Then $(\beta,\beta)=-(\beta,\gamma)$ and $2\beta+\gamma\in \Delta_+^\nca$. Since $\beta,\gamma \in \Delta_+$ and $(\beta,\gamma)<0$, we can choose another simple roots of $\Delta$ starting from $\{\beta,\gamma,\cdots\}$. Then according to the classifications of the Dynkin diagrams of the simple roots for Hermitian symmetric spaces in Lemma 4 of \cite{Wolf64}, this corresponds to type B and type C, since the Dynkin diagram does not depend on the choice of the simple roots.

In the case of type C, Lemma 4 of \cite{Wolf64} implies that the non-compact roots are of the form
$$\pm e_l\pm e_i,\, 1\le i\le l-1$$
where $l=\mathrm{rank }\,\mf g$. But then $(\beta,\beta)=(\gamma,\gamma)=2$ implies that $$k=-\frac{2(\beta,\gamma)}{(\beta,\beta)}=-\frac{2(\beta,\gamma)}{(\gamma,\gamma)}=1,$$
which is a contradiction.

Hence the case that $k=2$ corresponds only to type B. Then Lemma 4 of \cite{Wolf64} implies that we can choose a set of simple roots from $\Delta^\nc$. In particular, the set $\{\beta,\gamma,\cdots\}$ of simple roots chosen as above is a particular one, with the Dynkin diagram
$$
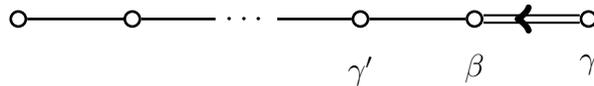
\begin{tikzpicture}[
    scale=1.5,
    line width=1pt,
    baseline=(current bounding box.center),
    every node/.style={circle, draw, fill=white, inner sep=2pt},
    label distance=2mm
  ]

  \node (alpha1) at (0, 0) [label=below:] {};
  \node (alpha2) at (1, 0) [label=below:] {};
  \node (dots) at (2, 0) [draw=none] {$\cdots$};
    \node (alpha3) at (3, 0) [label=below:$\gamma'$] {};
  \node (alpha_n-1) at (4, 0) [label=below:$\beta$] {};
  \node (alpha_n) at (5, 0) [label=below:$\gamma$] {};

  \draw (alpha1) -- (alpha2);
  \draw (alpha2) -- (dots);
  \draw (dots) -- (alpha3);
  \draw (alpha3) -- (alpha_n-1);

    \draw[double distance=2pt, line width=0.8pt, 
          postaction={decorate},
          decoration={markings, mark=at position 0.5 with {\arrow[scale=0.8]{<}}}] 
          (alpha_n-1) -- (alpha_n);
\end{tikzpicture}$$

Now if $\Delta_+^\ncb=\{\gamma\}$, then we can replace $\Delta_+^\nca$ with $\Delta_+^\ncb$ so that $\gamma \in \Delta_+^\nca$ and satisfies $(\gamma,\beta+\gamma)>0$, since $(\gamma,\gamma)>(\beta,\beta)$. This corresponds to the case of $Sp(4)$.

Otherwise, there exists $\gamma'\in \Delta_+^\ncb$ such that $(\beta,\gamma')<0$ with $$k'=-\frac{2(\beta,\gamma')}{(\beta,\beta)}=1.$$ Then the above argument implies that $\beta+\gamma'\in \Delta_+^\c$ and $(\beta,\beta+\gamma')>0$.
This proves the lemma for $\beta \in\Delta_+^{\nc,12}$ when $G_\R$ is simple.

Let $\beta\in \Delta_+^\nca\setminus \Delta_+^{\nc,12}$. Then $\beta+\gamma$ is not a root for any $\gamma \in \Delta_+^\nc$, which implies that $$(\beta,\gamma)\ge 0,\,\forall\, \gamma\in \Delta_+^\nc.$$

Let $\lambda=-\beta$ and consider it as a weight in $\mf h^*$. 
The key point of solving the conjecture of Green-Griffiths-Kerr is Proposition 2.2 in \cite{LiuShen24}. It asserts that for the set $$\Delta_+=\Delta_+^\c\cup \Delta_+^\nc$$ of positive roos, which defines a non-classical flag domain, there exists $\alpha\in \Delta_+^\c$ such that $(\lambda,\alpha)<0$, or $\gamma\in \Delta_+^\nc$ such that $(\lambda,\gamma)>0$.

Now $(\lambda,\gamma)=-(\beta,\gamma)\le 0$ for any $\gamma\in \Delta_+^\nc$. Hence there exists $\alpha\in \Delta_+^\c$ such that $-(\beta,\alpha)=(\lambda,\alpha)<0$.

Therefore we prove the lemma when $G_\R$ is simple.

In general, the flag domain factors as $$D=D^1\times \cdots \times D^n$$ with $D^i=G^i_\R/T^i$ and $G_\R^i$ simple $1\le i\le n$. Then we can take $$\Delta_+^\nca = {\Delta^1}_+^\nca \cup \cdots \cup {\Delta^n}_+^\nca$$ where each 
${\Delta^i}_+^\nca$ satisfies the properties of the lemma.
\end{proof}
\begin{remark}\label{beta alpga >0 remark}
\textnormal{(1) 
Lemma \ref{beta alpga >0} is a significant result concerning the root systems of non-classical flag domains. It builds upon the classification of Hermitian symmetric spaces and leverages our recent resolution of the conjecture by Green, Griffiths, and Kerr in \cite{LiuShen24}.}

\textnormal{(2) From the proof of Lemma \ref{beta alpga >0}, we see that, except for the case of $Sp(4)$, $\Delta_+^\ncb$ still satisfies Lemma \ref{beta alpga >0}, i.e. any $\beta\in \Delta_+^\ncb$, there exists an $\alpha\in \Delta_+^\c$ such that $(\beta,\alpha)>0$. 
Hence if $D''$ is diffeomorphic to $D$ with the complex structure given by
$${\Delta''}_+=\Delta_+^\c\cup\Delta_+^\nca\cup(-\Delta_+^\ncb),$$
so that $D''$ maps holomorphically onto the conjugate of the Hermitian symmetric space $G_\R/K$, and the results of the Penrose transformation on $D''$ still holds.}
\end{remark}

In fact, we have proved a stronger version of Lemma \ref{beta alpga >0}.
\begin{lemma}\label{beta alpga >0'}
Suppose that the flag domain $D$ factors as
$$
D = D^1 \times \cdots \times D^n,
$$ 
where $D^i = G_\R^i / T^i$ with $G_\R^i$ simple and of Hermitian type and $G_\R^i\neq Sp(4,\R)$ for $1\le i\le n$.
Then for any $\beta\in \Delta_+^\nc$, there exists an $\alpha\in \Delta_+^\c$ such that $(\beta,\alpha)>0$.
\end{lemma}


\section{Penrose transformation on compact quotients of flag domains}\label{section Pen compact}
In this section, we introduce the Penrose transformation on the compact quotients of the flag domains and prove that it is an isomorphism under the conditions given in Section \ref{section Pen} and the conditions induced by Property \textbf{W} of Williams. We also apply our result to the groups of automorphic forms on Hermitian symmetric domains, which can be identified with the higher automorphic cohomology groups of certain line bundles on the compact quotient of the non-classical flag domain.

Let $\Gamma\subset G_\R$ be a discrete, co-compact and neat subgroup. Then
$X=\Gamma\backslash D $ and $X'=\Gamma\backslash D'$ are compact complex manifolds, and $\mathscr I_\Gamma=\Gamma\backslash \mathscr I$, $\mathscr U_\Gamma=\Gamma\backslash \mathscr U$ and $\mathscr W_\Gamma=\Gamma\backslash \mathscr W$ are complex manifolds. Moreover diagram \eqref{Wto DD'} induces the following commutative diagram
\begin{equation}\label{Wto DD' compact}
  \xymatrix{
&\mathscr W_\Gamma\ar[ldd]_-{\pi}\ar[rdd]^-{\pi'}\ar[d]&\\
&\mathscr I_\Gamma\ar[ld]^-{\pi_{X}}\ar[rd]_-{\pi_{X'}}\\
X& &X'.}
\end{equation}

The homogenous line bundles $L_{\mu}=G_\R\times_T \C_{\mu}$ and $L_{\mu'}=G_\R\times_T \C_{\mu'}$ on $D$ and $D'$ are $G_\R$-invariant, and hence desecnd to the line bundles on $X$ and $X'$ respectively. We still denote them by $L_{\mu}$ and $L_{\mu'}$ for convenience. 

The $d_\pi$-closed form
$$\omega^\nca=\omega^{-\beta_1}\wedge \cdots \wedge \omega^{-\beta_q}\in \Gamma(\mathscr W, \Omega_\pi^q(L_{-2\rho_\nca}))$$
is left invariant, and descends to a form on $\mathscr W_\Gamma$ which we denote by 
$$\omega^\nca_\Gamma\in \Gamma(\mathscr W_\Gamma, \Omega_\pi^q(L_{-2\rho_\nca})).$$

In this section, we mainly consider the automorphic cohomology groups $H^*(X,L_\mu)$ and $H^*(X',L_{\mu'})$. From Page 236 of \cite{GGK}, we know that $\mathscr W_\Gamma=\Gamma\backslash \mathscr W$ is also Stein. Applying Theorem \ref{EGW} of EGW again, we have that
\begin{eqnarray*}
  H^*(X,L_\mu) &\cong & H^*_{DR}(\mathscr W_\Gamma,\Omega_{\pi}^\bullet(L_{\mu})), \\
  H^*(X',L_{\mu'})&\cong & H^*_{DR}(\mathscr W_\Gamma,\Omega_{\pi'}^\bullet(L_{\mu'})).
\end{eqnarray*}

It is important to note that the arguments for the main results in Section \ref{section Pen} are local. Consequently, we introduce the following definition of the Penrose transformation $\mathscr{P}$ on the compact quotients of flag domains and present the corresponding theorem establishing the injectivity of $\mathscr{P}$.

\begin{lemdef}
Let $\mu$ and $\mu'$ be two weights such that $\mu+\rho=\mu'+\rho'$.
Then the Penrose transformation on automorphic cohomology groups $$\mathscr P:\, H^0(X',L_{\mu'})\to H^q(X,L_{\mu})$$ is defined by
\begin{equation}\label{PenD compact}
  \xymatrix{
  H^0(X',L_{\mu'}) \ar[d]^-{\cong}\ar[rr]^-{\mathscr P}&& H^q(X,L_{\mu})\ar[d]^-{\cong}\\
  H^0_{DR}(\mathscr W_\Gamma,\Omega_{\pi'}^\bullet(L_{\mu'}))\ar[rr]^-{\omega^\nca_\Gamma}&&H^q_{DR}(\mathscr W_\Gamma,\Omega_{\pi}^\bullet(L_{\mu})).
  }
\end{equation}
\end{lemdef}
Similar to Theorem \ref{mainD}, we have the following theorem.
\begin{theorem}\label{mainD compact1}
Let $D=G_\R/T$ be a non-classical flag domain with $G_\R$ of Hermitian type. Let $D'$ be the classical flag domain which is diffeomorphic to $D$.
Let $X=\Gamma\backslash D$ and $X'=\Gamma\backslash D'$ be their corresponding  quotients by a discrete, co-compact and neat subgroup $\Gamma\subset G_\R$.

If $\mu$ and $\mu'$ are two weights with $\mu+\rho=\mu'+\rho'$, and satisfy that for any $\beta\in\Delta_+^\nca$ there exists $\alpha_\beta\in \Delta_+^\c$ such that 
\begin{equation}
  (\alpha_\beta,\mu'-\beta)<0, \tag{\ref{Pen inj conditions}}
\end{equation}
Then the Penrose transformation on automorphic cohomology groups $$\mathscr P:\, H^0(X',L_{\mu'})\to H^q(X,L_{\mu})$$ given by \eqref{PenD compact} is injective.
\end{theorem}

As examples of the conditions in Theorem \ref{mainD compact1}, we provide the following theorem.
\begin{theorem}\label{mainD compact example}
Let the notations be as Theorem \ref{mainD compact1}.
If $\mu$ and $\mu'$ are two weights such that $\mu+\rho=\mu'+\rho'$ and 
$$(\mu',\alpha)= 0,\, \forall\, \alpha\in \Delta_+^\c.$$
Then the Penrose transformation on automorphic cohomology groups $$\mathscr P:\, H^0(X',L_{\mu'})\to H^q(X,L_{\mu})$$ given by \eqref{PenD compact} is injective.
\end{theorem}

We are interested in the case that the Penrose transformation 
\begin{equation}\label{Pen isom}
  \xymatrix{\mathscr P:\, H^0(X',L_{\mu'})\ar[r]^-{\cong}& H^q(X,L_{\mu})}
\end{equation}
is an isomorphism. Under Theorem \ref{mainD compact1}, \eqref{Pen isom} is reduced to proving that 
\begin{equation}\label{dimX=dimX'}
  \dim_\C H^0(X',L_{\mu'})=\dim_\C H^q(X,L_{\mu}).
\end{equation}

To get a natural condition for \eqref{dimX=dimX'}, we introduce the work of Schmid, Griffiths, Williams and et al. on the geometric realizations of representations of semisimple real Lie groups. Good references for this survey are Schmid's paper \cite{schmid97} and Lecture 5 and Lecture 9 of \cite{GGK}. One is referred to \cite{knapp86} for the basic notions in this area and more details.

First we fix some notations. Let $\widehat{G}_\R$ be the set of equivalence classes of irreducible unitary representations 
$$\pi:\, G_\R \to \text{Aut}(\tilde V_\pi)$$ of $G_\R $ on a Hilbert space $\tilde V_\pi$. Then the subspace $V_\pi=\tilde V_{\pi,K-\mathrm{finite}}$ of $K$-finite vectors in $\tilde V_\pi$ is a Harish-Chandra module, which is called the Harish-Chandra module associated to the unitary representation $\tilde V_\pi$. We also denote the unitary representation $ \tilde V_\pi$ and its associated Harish-Chandra module  $V_\pi$ by $\tilde V_\zeta$ and $V_\zeta$ respectively, if $V_\pi$ has infinitesimal character $\chi_\zeta$.

The $L^2$-cohomology group $H_{(2)}^k(D,L_\mu)$ of the homogeneous line bundle $L_\mu$ on $D$ can be realized as 
$$H_{(2)}^k(D,L_\mu) =\int_{\widehat{G}_\R}\tilde{V_\pi}\widehat{\otimes}H^k(\mf n_+,\tilde{V_\pi}^*)_{-\mu}$$
via the Plancherel decomposition 
$$L^2(G_\R)= \int_{\widehat{G}_\R}\tilde{V_\pi}\widehat{\otimes}\tilde{V_\pi}^*d\pi,$$
where $\tilde{V_\pi}^*$ is the dual space of $\tilde{V_\pi}$, and $d\pi$ is the Plancherel measure which assigns positive mass to the discrete series representations. Here $H^k(\mf n_+,\tilde{V_\pi}^*)$ is the $\mf n$-cohomology and $H^k(\mf n_+,\tilde{V_\pi}^*)_{-\mu}$ is the set of elements in $H^k(\mf n_+,\tilde{V_\pi}^*)$ on which $\mf h$ acts by the weight $-\mu$.
 
The celebrated work of Schmid on the solution to the Langlands conjecture can be summarized as follows.

\begin{theorem}[Schmid]\label{Schmid thm}
Let $L_\mu$ be homogeneous line bundle on the flag domain $D=G_\R/T$ defined by a weight $\mu$. Then
we have the following realization of the representation of $G_\R$.

(1) When $\mu+\rho$ singular, i.e. $\exists\, \alpha\in \Delta_+$ s.t. $(\mu+\rho,\alpha)=0$, the $L^2$-cohomology group $H_{(2)}^k(D,L_\mu)$ is zero;

(2) When $\mu+\rho$ is regular, i.e. $(\mu+\rho,\alpha)\neq 0$ for any $\alpha\in \Delta_+$, the $L^2$-cohomology group $H_{(2)}^k(D,L_\mu)$ is zero if $k\neq q(\mu+\rho)$, and $H_{(2)}^{q(\mu+\rho)}(D,L_\mu)$ is a discrete series representation of $G_\R$, whose associated Harish-Chandra module has infinitesimal character $\chi_{\mu+\rho}$;

(3) When $\mu+\rho$ belongs to the anti-dominant Weyl chamber, the cohomology group $H^d(D,L_\mu)$ is a Harish-Chandra module with infinitesimal character $\chi_{\mu+\rho}$, where $d=\dim_\C Z_o$.
\end{theorem}

The automorphic cohomology $H^k(X, L_\mu)$ on the compact quotient $X=\Gamma\backslash D$ can be realized as the $\mf n$-cohomology:
$$H^k(X, L_\mu) = \bigoplus_{\pi \in \widehat{G}_\R} H^k(\mf n_+, V_\pi)_{-\mu}^{\oplus m_\pi(\Gamma)}$$
where $m_\pi(\Gamma)$ is the multiplicity of the associated Harish-Chandra module $V_\pi$ in $ L^2(\Gamma \backslash G_\R)$.
Therefore $\dim_\C H^k(X, L_\mu)$ is determined by $$\dim_\C H^k(\mf n_+, V_\pi)^{\oplus m_\pi(\Gamma)}, \ \ \pi \in \widehat{G}_\R.$$

For the $\mf n$-cohomology, we have the following Williams Lemma from \cite{Wi1} as introduced in Lecture 8 of \cite{GGK}.

\begin{lemma}[Williams Lemma]\label{W lemma}
Let $\tilde V$ be an irreducible unitary representation of $G_\R$ and $V$ be its associated Harish-Chandra module. If $ \mu $ is a weight satisfying
\begin{enumerate}
    \item $\mu + \rho$ is regular;
    \item Property \textbf{W}: For each $ \beta \in \Delta^\nc $ with $ (\mu + \rho, \beta) > 0 $, we have that
    $$
    \left( \mu + \rho - \frac{1}{2} \sum_{\begin{subarray}{c}\alpha \in \Delta\\ (\mu + \rho, \alpha) > 0\end{subarray}}  \alpha,\, \beta\right)  > 0.
    $$
\end{enumerate}
Then $H^k(\mf n_+, V)_{-\mu} \neq 0$ implies that 
$$\begin{cases}
    k = q(\mu + \rho) \\
    \dim_\C H^k(\mf n_+, V)_{-\mu} = 1\\
    \tilde V= \tilde V_{-(\mu + \rho)}
\end{cases}$$
where $\tilde V_{-(\mu + \rho)}$ is a discrete series representation with infinitesimal character $ \chi_{-(\mu + \rho)}$.
\end{lemma}
Another result we need is the following theorem.
\begin{theorem}[Williams, Theorem 2.4 in \cite{Wi2}]\label{W thm}
Let $X=\Gamma\backslash D$ ba a compact quotient of the flag domain $D=G_\R/T$ and $L_\mu$ be a locally homogenous line bundle on $X$ defined by the weight $\mu$. Assume that $\mu + \rho$ is regular and satisfies Property \textbf{W} of Williams in Lemma \ref{W lemma}. Then 
$$\dim_\C H^{q(\mu + \rho)}(X,L_\mu)=m_{-(\mu + \rho)}(\Gamma),$$
the multiplicity of $V_{-(\mu + \rho)}$ in $ L^2(\Gamma \backslash G_\R)$, and $H^{k}(X,L_\mu)=0$ for $k\neq q(\mu + \rho)$.
\end{theorem}
With the above preparations we can prove the following theorem.
\begin{theorem}\label{mainD compact2}
Let $D=G_\R/T$ be a non-classical flag domain with $G_\R$ of Hermitian type. Let $D'$ be the classical flag domain which is diffeomorphic to $D$.
Let $X=\Gamma\backslash D$ and $X'=\Gamma\backslash D'$ be their corresponding quotients by a discrete, co-compact and neat subgroup $\Gamma\subset G_\R$.

Let $\mu$ and $\mu'$ be two weights such that $\mu+\rho=\mu'+\rho'$ is regular and 
$\mu'+\rho'$ 
lies in the Weyl chamber determined by the set of roots
\begin{equation}\label{chamber}
  \Delta_+^\c \cup \Delta_+^\nca\cup\left(-\Delta_+^\ncb \right)=\Delta_+^\c \cup\left(-{\Delta'}_+^\nc \right).
\end{equation}
If for any $\beta\in\Delta_+^\nca$ there exists an $\alpha_\beta\in \Delta_+^\c$ such that 
\begin{equation}
  (\alpha_\beta,\mu'-\beta)<0, \tag{\ref{Pen inj conditions}}
\end{equation}
and moreover,
\begin{equation}\label{PW}
 (\mu'+2\rho'_\nc,\, -{\Delta'}_+^\nc )>0,
\end{equation}
then the Penrose transformation $$\mathscr P:\, H^0(X',L_{\mu'})\to H^q(X,L_{\mu})$$ given by \eqref{PenD compact} is an isomorphism.
\end{theorem}
\begin{proof}
Since $$\zeta=\mu+\rho=\mu'+\rho'$$ is regular and lies in the Weyl chamber determined by the set of roots in \eqref{chamber}, we have that
\begin{equation}\label{chamber equiv}
  (\zeta,\alpha)>0,\, \forall \, \alpha\in \Delta_+^\c;\, (\zeta,\beta)>0,\, \forall \, \beta\in \Delta_+^\nca;\,(\zeta,\gamma)<0,\, \forall \, \gamma\in \Delta_+^\ncb,
\end{equation}
which implies that $$q'(\mu'+\rho')=0,\, q(\mu+\rho)=q$$
where $q=\#\Delta_+^\nca$.

From Theorem \ref{mainD compact1}, the Penrose transformation $$\mathscr P:\, H^0(X',L_{\mu'})\to H^q(X,L_{\mu})$$ is injective. Hence we need to prove that $$\dim_\C H^0(X',L_{\mu'})=\dim_\C H^q(X,L_{\mu}).$$
For this we only need to check that conditions in \eqref{PW} is equivalent to Property \textbf{W} of Williams in Lemma \ref{W lemma}, since then Theorem \ref{W thm} implies that
$$\dim_\C H^0(X',L_{\mu'})=m_{-(\mu' + \rho')}(\Gamma)=m_{-(\mu + \rho)}(\Gamma)=\dim_\C H^q(X,L_{\mu}).$$

Note that under \eqref{chamber equiv} the conditions in Property \textbf{W}, we have 
$$\beta \in \Delta^\nc \text{ with } (\mu + \rho, \beta) > 0 \Longleftrightarrow \beta\in  \Delta_+^\nca\cup\left(-\Delta_+^\ncb \right)=-{\Delta'}_+^\nc $$
and 
\begin{eqnarray*}
  \frac{1}{2}\sum\limits_{\begin{subarray}{c}\alpha \in \Delta\\ (\mu + \rho, \alpha) > 0\end{subarray}}  \alpha &=& \frac{1}{2} \sum\limits_{\alpha\in \Delta_+^\c}\alpha+\frac{1}{2}\sum\limits_{\beta\in \Delta_+^\nca}\beta-\frac{1}{2}\sum\limits_{\gamma\in \Delta_+^\ncb}\gamma  \\
    &=& \rho_\c+\rho_\nca-\rho_\ncb =\rho_\c-\rho'_\nc,
\end{eqnarray*}
which implies that 
$$\mu + \rho - \frac{1}{2} \sum_{\begin{subarray}{c}\alpha \in \Delta\\ (\mu + \rho, \alpha) > 0\end{subarray}}  \alpha =\mu+2\rho_\ncb=\mu'+2\rho'_\nc.$$
Hence \eqref{PW} is equivalent to Property \textbf{W}.
\end{proof}

Due to its significance in arithmetic geometry and number theory, we consider the homogenous line bundle $L_{\mu'_c}$ defined as the pull-back of the canonical bundle $\omega_{\mathbb B}$ via the holomorphic projection map
$$p':\, D'\to G_\R/K\cong \mathbb B,$$
which is still denoted by $L_{\mu'_c}=\omega_{\mathbb B}\to D'$.

Since $$\mathrm{T}^{1,0}_o G_\R/K \cong \mf p'_-,$$ we have that
$$\mu'_c=-\sum_{\beta \in {\Delta'}_+^\nc}\beta= 2\rho_\nca-2\rho_\ncb.$$ 
Let $k_0$ the maximal positive integer such that $\mu'_{c0}\triangleq \mu'_c/k_0$ is still a weight.
Then 
$$L_{\mu'_{c0}}=\omega_{\mathbb B}^{\otimes 1/k_0}\to D'$$
is a well-defined line bundle on $D'$. 
Let $$\mu_k'=k\mu'_{c0}=\frac{k}{k_0}(2\rho_\nca-2\rho_\ncb),\,\mu_k=\mu_k'-2\rho_\nca.$$ Then
$$L_{\mu_k'}=\omega_{\mathbb B}^{\otimes k/k_0}\to D',\, L_{\mu_k}\to D$$
are two line bundles whose cohomology groups and automorphic cohomology groups are Penrose related.

\begin{theorem}\label{automorphic theorem}
Let the notations be as above. Then there exists a positive integer $N$ such that for $k\ge N$, the Penrose transformation  on automorphic cohomology groups
\begin{equation}\label{PT automorphic form}
 \mathscr P:\, H^0(X',\omega_{\mathbb B}^{\otimes k/k_0})\to H^q(X,L_{\mu_k})
\end{equation}
is an isomorphism. Therefore $H^q(X,L_{\mu_k})$ is isomorphic to the group $$H^0(\Gamma\backslash \mathbb B,\omega_{\mathbb B}^{\otimes k/k_0})$$ of automorphic forms on $\mathbb B$.
\end{theorem}
\begin{proof}
The proof is divided into the following steps.

\noindent \textbf{Step 1.} We first prove that $$(2\rho_\nca-2\rho_\ncb,\alpha)=-\sum_{\beta\in {\Delta'}_+^\nc}\left(\beta,\alpha\right)=0,\,\forall \, \alpha\in \Delta_+^\c .$$
Then $(\mu_k',\alpha)=0$, for any $\alpha\in \Delta_+^\c$ and Theorem \ref{mainD compact example} implies that the Penrose transformation \eqref{PT automorphic form} is injective.

First note that $$[\mf k_\pm,\mf p'_-]\subset \mf p'_-\Longleftrightarrow \forall\, \alpha\in \Delta_+^\c,\beta\in {\Delta'}_+^\nc, \beta\pm\alpha \in {\Delta'}_+^\nc \text{ or is not a root}.$$

Now we fix any $\alpha \in \Delta_+^\c$. If there exists $\beta\in {\Delta'}_+^\nc$ such that $(\beta,\alpha)\neq 0$, then we have either that $(\beta,\alpha)>0$, which implies that 
$$\beta,\beta-\alpha,\cdots,\beta-k \alpha \text{ are roots in } {\Delta'}_+^\nc,\, \text{where } k=\frac{2(\beta,\alpha)}{(\alpha,\alpha)},$$
and 
$$\left(\beta+(\beta-\alpha)+\cdots+(\beta- k\alpha),\alpha\right)=(k+1)\left(\beta,\alpha\right)-\frac{k(k+1)}{2}\left(\alpha,\alpha\right)=0,$$
or that 
$(\beta,\alpha)<0$, which implies that 
$$\beta,\beta+\alpha,\cdots,\beta+k \alpha \text{ are roots in } {\Delta'}_+^\nc,\, \text{where } k=\frac{-2(\beta,\alpha)}{(\alpha,\alpha)},$$
and 
$$\left(\beta+(\beta+\alpha)+\cdots+(\beta+ k\alpha),\alpha\right)=(k+1)\left(\beta,\alpha\right)+\frac{k(k+1)}{2}\left(\alpha,\alpha\right)=0.$$
Therefore 
$$ \sum_{\substack{\beta\in {\Delta'}_+^\nc\\ (\beta,\alpha)\neq 0}}\left(\beta,\alpha\right)=0$$
and
$$(2\rho_\nca-2\rho_\ncb,\alpha)=-\sum_{\substack{\beta\in {\Delta'}_+^\nc\\ (\beta,\alpha)= 0}}\left(\beta,\alpha\right)-\sum_{\substack{\beta\in {\Delta'}_+^\nc\\ (\beta,\alpha)\neq 0}}\left(\beta,\alpha\right)=0.$$

\noindent\textbf{Step 2.} Lemma \ref{Delta' positive} implies that
$$(\beta,\beta')\ge 0,\,(\beta,\gamma)\le 0,\,(\gamma,\gamma')\ge 0,\,\forall \, \beta,\beta'\in \Delta_+^\nca,\forall \, \gamma,\gamma'\in \Delta_+^\ncb.$$
Hence 
\begin{eqnarray*}
  (2\rho_\nca-2\rho_\ncb,\beta) &=& \left(\beta+\sum_{{\beta'\in \Delta_+^\nca,\, \beta'\neq \beta}}\beta'-\sum_{{\gamma\in \Delta_+^\ncb}}\gamma,\, \beta\right) \\
   &\ge&(\beta,\beta) >0 ,\,\forall \, \beta\in \Delta_+^\nca
\end{eqnarray*}
and similarly
$$(2\rho_\nca-2\rho_\ncb,\gamma)\le -(\gamma,\gamma)<0,\,\forall \, \gamma\in \Delta_+^\ncb.$$

\noindent\textbf{Step 3.} For $$\mu_k'+\rho'=\rho_\c+\left(\frac{k}{k_0}-\frac{1}{2}\right)(2\rho_\nca-2\rho_\ncb),$$
Step 1 and Step 2 imply that there exists a positive integer $N_1$ such that for $k\ge N_1$,
\begin{eqnarray}
  && (\mu_k'+\rho',\alpha)=(\rho_\c,\alpha)>0,\,\forall\, \alpha\in \Delta_+^\c; \label{positive N_1 1}\\
  &&(\mu_k'+\rho',\beta)=\left(\frac{k}{k_0}-\frac{1}{2}\right)(2\rho_\nca-2\rho_\ncb,\beta)+(\rho_\c,\beta)>0,\,\forall \, \beta\in \Delta_+^\nca; \label{positive N_1 2}\\
  && (\mu_k'+\rho',\gamma)=\left(\frac{k}{k_0}-\frac{1}{2}\right)(2\rho_\nca-2\rho_\ncb,\gamma)+(\rho_\c,\gamma)<0,\,\forall \, \gamma\in \Delta_+^\ncb.  \label{positive N_1 3}
\end{eqnarray}
This implies that $\mu_k+\rho=\mu_k'+\rho'$ is regular and lies in the Weyl chamber determined by the set of roots in \eqref{chamber} in Theorem \ref{mainD compact2}.

Similarly, there exists a positive integer $N$ with $N\ge N_1$ such that for $k\ge N$,
\begin{equation}\label{positive N}
  (\mu_k'+2\rho'_\nc,\, -{\Delta'}_+^\nc )=\left(\frac{k}{k_0}-1\right)\left(2\rho_\nca-2\rho_\ncb,{\Delta}_+^\nca\cup\left(-\Delta_+^\ncb\right)\right)>0.
\end{equation}
Hence \eqref{PW} is satisfied.

Finally, by Theorem \ref{mainD compact2}, Step 1 and Step 3 imply that the Penrose transformation \eqref{PT automorphic form} is an isomorphism.
\end{proof}

\section{Cup products of the automorphic cohomology groups}\label{cup-product section}
In this section, we apply our results to study the cup products of automorphic cohomology groups on $D$, with the target being a TDLDS.

We take $\Gamma\subset G_\R$ a discrete, co-compact and neat subgroup.
First we prove a generalized version of Theorem \ref{automorphic theorem}.
\begin{theorem}\label{twist automorphic theorem}
Let $\mu_0$ be a weight satisfying that
\begin{equation}\label{mu0>0}
 (\mu_0,\alpha)\ge 0,\,\forall\, \alpha\in \Delta_+^\c,
\end{equation}
and that for any $\beta\in \Delta_+^\nca$ there exists an $\alpha_\beta\in \Delta_+^\c$ such that 
\begin{equation}\label{mu0 le 0}
 (\mu_0-\beta,\alpha_\beta)< 0.
\end{equation}
Let $L_{\mu'_{k,0}}\to D'$ and $L_{\mu_{k,0}}\to D$ be two homogenous line bundles defined by the weights
$$\mu'_{k,0}=\frac{k}{k_0}(2\rho_\nca-2\rho_\ncb)+\mu_0,\,\mu_{k,0}=\mu'_{k,0}-2\rho_\nca.$$
Then there exists a positive integer $N$ such that for $k\ge N$, the Penrose transformation  on automorphic cohomology groups
\begin{equation}\label{twist PT automorphic form}
 \mathscr P:\, H^0(X',L_{\mu'_{k,0}})\to H^q(X,L_{\mu_{k,0}})
\end{equation}
is an isomorphism.
\end{theorem}
\begin{proof}
The proof is similar to that of Theorem \ref{automorphic theorem}.

From Theorem \ref{mainD compact2}, the Penrose transformation \eqref{twist PT automorphic form} is injective, provided that for any $\beta\in\Delta_+^\nca$ there exists an $\alpha_\beta\in \Delta_+^\c$ such that 
\begin{equation}
  (\mu'_{k,0}-\beta,\alpha_\beta)=(\mu_{0}-\beta,\alpha_\beta)< 0, 
\end{equation}
which holds from Step 1 of the proof of Theorem \ref{automorphic theorem} and the assumption \eqref{mu0 le 0}.

Note that
$$\mu'_{k,0}+\rho'=\mu_0+\rho_\c+\left(\frac{k}{k_0}-\frac{1}{2}\right)(2\rho_\nca-2\rho_\ncb),$$
Step 1 and Step 2 in Theorem \ref{automorphic theorem} and the assumption \eqref{mu0>0} imply that there exists a positive integer $N_1$ such that for $k\ge N_1$,
\begin{eqnarray*}
  && (\mu'_{k,0}+\rho',\alpha)=(\mu_0+\rho_\c,\alpha)\ge (\rho_\c,\alpha)>0,\,\forall\, \alpha\in \Delta_+^\c; \label{positive N_1 1}\\
  &&(\mu'_{k,0}+\rho',\beta)=\left(\frac{k}{k_0}-\frac{1}{2}\right)(2\rho_\nca-2\rho_\ncb,\beta)+(\mu_0+\rho_\c,\beta)>0,\,\forall \, \beta\in \Delta_+^\nca; \label{positive N_1 2}\\
  && (\mu'_{k,0}+\rho',\gamma)=\left(\frac{k}{k_0}-\frac{1}{2}\right)(2\rho_\nca-2\rho_\ncb,\gamma)+(\mu_0+\rho_\c,\gamma)<0,\,\forall \, \gamma\in \Delta_+^\ncb.  \label{positive N_1 3}
\end{eqnarray*}
This implies that $\mu_{k,0}+\rho=\mu'_{k,0}+\rho'$ is regular and lies in the Weyl chamber determined by the set of roots in \eqref{chamber} in Theorem \ref{mainD compact2}.

Similarly, there exists a positive integer $N$ with $N\ge N_1$ such that for $k\ge N$,
\begin{eqnarray*}\label{positive N}
  (\mu'_{k,0}+2\rho'_\nc,\, -{\Delta'}_+^\nc )&=&\left(\frac{k}{k_0}-1\right)\left(2\rho_\nca-2\rho_\ncb,{\Delta}_+^\nca\cup\left(-\Delta_+^\ncb\right)\right)
  \\
  && +\left(\mu_0, {\Delta}_+^\nca\cup\left(-\Delta_+^\ncb\right)\right)>0.
\end{eqnarray*}
Hence \eqref{PW} in Theorem \ref{mainD compact2} is satisfied.

Finally, by Theorem \ref{mainD compact2}, the Penrose transformation \eqref{PT automorphic form} is an isomorphism.
\end{proof}

Let $D$ be a non-classical flag domain with $G_\R$ of Hermitian type. Let 
$$\Delta_+=\Delta_+^\c\cup\Delta_+^\nca\cup\Delta_+^\ncb $$
be the set of positive roots corresponding to the complex structure of $D$.

Let $D'$ and $D''$ be the classical flag domains, which are diffeomorphic to $D$, with the complex structures determined respectively by the set of positive roots
\begin{eqnarray*}
  {\Delta'}_+=\Delta_+^\c\cup(-\Delta_+^\nca)\cup\Delta_+^\ncb;\\
  {\Delta''}_+=\Delta_+^\c\cup\Delta_+^\nca\cup(-\Delta_+^\ncb).
\end{eqnarray*}
Then we have two holomorphic projection maps
$$p':\, D'\to \mathbb B,\, p'':\, D''\to \bar{\mathbb B},$$
where $\bar{\mathbb B}$ denotes the complex conjugate of the Hermitian symmetric domain $\mathbb B\cong G_\R/K$, with the complex structure given by ${\Delta'}_+$.

Let 
$$L_{\mu_{k,0}'}=\omega_{\mathbb B}^{\otimes k/k_0}\otimes L_{\mu_0}\to \mathbb B,\, L_{\lambda'_{k,0}}=\omega_{\bar{\mathbb B}}^{\otimes k/k_0}\otimes L_{\lambda_0}\to \bar{\mathbb B}$$
be the canonical bundles tensored with the homogenous line bundles defined by the weights $\mu_0$ and $\lambda_0$ with 
$$\mu_0+\lambda_0=\rho_\nc-\rho_\c.$$
Hence 
$$\mu_{k,0}'=\frac{k}{k_0}(2\rho_\nca-2\rho_\ncb)+\mu_0,\, \lambda_{k,0}'=-\frac{k}{k_0}(2\rho_\nca-2\rho_\ncb)+\lambda_0.$$

We still denote their pull-backs to $D'$ and $D''$ by $L_{\mu_{k,0}'}\to D'$ and $L_{\lambda'_{k,0}}\to D''$ respectively. Then $H^0(D',L_{\mu_{k,0}'})$ and $H^0(D'',L_{\lambda'_{k,0}})$ are Penrose related to $H^q(D,L_{\mu_{k,0}})$ and $H^p(D, L_{\lambda_{k,0}})$ respectively, where 
$$\mu_{k,0}= \mu_{k,0}'-2\rho_\nca,\, \lambda_{k,0}=\lambda_{k,0}'-2\rho_\ncb,$$
and $q=\#\Delta_+^\nca$, $p=\#\Delta_+^\ncb$.

\begin{theorem}\label{cup-products}
Let $X$ be a smooth and compact quotient of a non-classical flag domain $D$ with $G_\R$ of Hermitian type.
Let $\mu_0,\lambda_0$ be two weights with $\mu_0+\lambda_0=\rho_\nc-\rho_\c$.
Assume that
\begin{equation}\label{mu0 lambda_{0} alpha >=0}
 (\mu_0,\alpha), (\lambda_0,\alpha)\ge 0,\,\forall\, \alpha\in \Delta_+^\c,
\end{equation}
and that for any $\beta\in \Delta_+^\nca$ and any $\gamma\in \Delta_+^\ncb$ there exist  $\alpha_\beta$ and $\alpha_\gamma$ in $\Delta_+^\c$ such that 
\begin{equation}\label{mu0 lambda_{0} beta <0}
 (\mu_0-\beta,\alpha_\beta)< 0, \text{ and }(\lambda_0-\gamma,\alpha_\gamma)< 0.
\end{equation}
Then there exists a homomorphism of groups
\begin{equation}\label{cup-product1}
H^0(\Gamma\backslash\mathbb B,\omega_{\mathbb B}^{\otimes k/k_0}\otimes L_{\mu_0})\times H^0(\Gamma\backslash \bar{\mathbb B}, \omega_{\bar{\mathbb B}}^{\otimes k/k_0}\otimes L_{\lambda_0}) \to H^{d}(X,L_{-\rho})^*,
\end{equation}
for all $k$ greater than some positive integer $N$, where $d=\#\Delta_+^\c$.
\end{theorem}
\begin{proof}
From Theorem \ref{twist automorphic theorem}, we have the isomorphisms
\begin{eqnarray*}
&&H^0(\Gamma\backslash\mathbb B,\omega_{\mathbb B}^{\otimes k/k_0}\otimes L_{\mu_0})\cong H^q(X,L_{\mu_{k,0}})\\
&&H^0(\Gamma\backslash \bar{\mathbb B}, \omega_{\bar{\mathbb B}}^{\otimes k/k_0}\otimes L_{\lambda_0}) \cong H^p(X, L_{\lambda_{k,0}}).
\end{eqnarray*}
Since $\omega_X=L_{-2\rho}$ and
$$\mu_{k,0}+\lambda_{k,0}= \mu_{k,0}'+\lambda_{k,0}'-2\rho_\nca-2\rho_\ncb=\mu_0+\lambda_0-2\rho_\nc=-\rho,$$
the cup-product of the automorphic cohomology groups 
\begin{equation}\label{cup-product2}
H^q(X,L_{\mu_{k,0}}) \times H^p(X, L_{\lambda_{k,0}}) \to H^{p+q}(X,L_{\mu_{k,0}+\lambda_{k,0}})
\end{equation}
together with the Serre duality $$H^{p+q}(X,L_{\mu_{k,0}+\lambda_{k,0}})=H^{\dim_\C X-p-q}(X,\omega_X \otimes L_{-\mu_{k,0}-\lambda_{k,0}})^*=H^{d}(X,L_{-\rho})^*$$
gives the homomorphism of groups in \eqref{cup-product1}.
\end{proof}
\begin{remark}
\textnormal{The conditions with \eqref{mu0 lambda_{0} alpha >=0} and \eqref{mu0 lambda_{0} beta <0} satisfied simultaneously are quite strict. We refer the reader to Section \ref{examples} for the examples of $SU(r,s)$ and $Sp(2n,\C)$, where only $SU(s+1,s)$ and $Sp(6,\C)$ are satisfied.
}
\end{remark}
\begin{corollary}\label{cup-products2}
Let the notations be as Theorem \ref{cup-products} and suppose that the flag domain $D$ factors as
$$
D = D^1 \times \cdots \times D^n,
$$ 
where $D^i = G_\R^i / T^i$ with $G_\R^i$ simple and of Hermitian type and $G_\R^i\neq Sp(4,\R)$ for $1\le i\le n$.
If
$$\rho_\nc-\rho_\c=0,$$
then there exists a homomorphism of groups 
\begin{equation}\label{cup-product1'}
H^0(\Gamma\backslash\mathbb B,\omega_{\mathbb B}^{\otimes k/k_0})\times H^0(\Gamma\backslash \bar{\mathbb B}, \omega_{\bar{\mathbb B}}^{\otimes k/k_0}) \to H^{d}(X,L_{-\rho})^*,
\end{equation}
for all $k$ greater than some positive integer $N$.
\end{corollary}
\begin{proof}
We take $\mu_0=\lambda_0=0$. Then \eqref{mu0 lambda_{0} alpha >=0} holds automatically and \eqref{mu0 lambda_{0} beta <0} holds from Lemma \ref{beta alpga >0'}. Hence the corollary follows from Theorem \ref{cup-products}.
\end{proof}

\begin{remark}
\textnormal{The cohomology group $ H^{d}(X, L_{-\rho}) $ is a Harish-Chandra module with infinitesimal character $ \chi_{0} $, as stated in Schmid's Theorem, see (3) of Theorem \ref{Schmid thm}. This is called TDLDS (totally degenerate limits of discrete series) in \cite{GGK}. Please see \cite{CK07} and \cite{GGK} for further details on the applications of TDLDS to representation theory and Hodge theory.}
\end{remark}

Therefore, it is of interest to propose the following problem for the non-classical flag domain $ D $ with $ G_\R $ of Hermitian type.

\begin{problem}
(1) When are the homomorphisms of groups in \eqref{cup-product1} and \eqref{cup-product1'} surjective?

(2) When do the kernels of the homomorphisms of groups in \eqref{cup-product1} and \eqref{cup-product1'} have arithmetic structures?
\end{problem}

From the proof of Theorem \ref{cup-products}, we see that Problem (1) is equivalent to the injectivity of the homomorphism
\begin{equation*}
H^q(X,L_{\mu_{k,0}}) \times  H^{d}(X,L_{-\rho})\to H^{q+d}(X, L_{\mu_{k,0}-\rho}).
\end{equation*}

To our knowledge, Problem (1) is solved only for $SU(2,1)$ and $Sp(4)$. Please see \cite{C1}, \cite{C2}, \cite{C3}, \cite{GGK} and \cite{Ker14}. Problem (2) seems to be unknown.

We will provide detailed study of general cases of Problem (1) in our forthcoming paper.

\section{Examples}\label{examples}
In this section, we discuss the examples from the classical groups $SU(r,s)$ and $Sp(2n,\C)$ to illustrate the results proved in previous sections. The basic notions of Mumford-Tate groups and domains are needed to understand the examples. Please see \cite{GGK12} and \cite{GGK}, Lecture 3 and Lecture 5, for details.

\begin{example}[$SU(r,s),\, r\ge s$]\label{U(r,s)}
Let $V$ be a $\mathbb{Q}$-vector space of dimension $2(r+s)$ with an action of $\mathbb F = \mathbb{Q}(\sqrt{-1})$. Let $\vec{v}=(v_1,\cdots,v_{2(r+s)})^T$ be a basis of $V$ such that $\sqrt{-1}.v_i=v_{i+r+s}$ for $i=1,\cdots,r+s$. Then $\sqrt{-1}.v_{i+r+s}=-v_i$, and the action of $\sqrt{-1}$ on the column vector $\vec{v}$ is given by
\begin{equation}\label{action of -1}
  \sqrt{-1}.\vec{v}=\begin{pmatrix}
                      & I \\
                     I &  
                   \end{pmatrix}\vec{v}.
\end{equation}

Let $V_{\mathbb F} ={\mathbb F} \otimes_{\mathbb{Q}} V$. Then we have the decomposition
$$V_{\mathbb F} = V_+ \oplus V_-$$
where $$V_\pm=\mathbb F\{v_i\mp\sqrt{-1}v_{i+r+s}:\,i=1,\cdots,r+s\}$$ is the $\pm$ eigenspaces of $V_{\mathbb F}$ with respect to the action of ${\mathbb F}$. One should be careful to distinguish the scalar product $\sqrt{-1}v$ and the action $\sqrt{-1}.v$ of $\sqrt{-1}$ on $v\in V_{\mathbb F}$.

Let $Q$ be a non-degenerate anti-symmetric bilinear form $Q :\, V \otimes V \to \mathbb{Q}$. Let $$H(v_i,v_j)=\sqrt{-1}Q(v_i,\bar{v_j}),$$ whose matrix $H(v_i,v_j)$ under the basis $v$ is 
$$H_v =\frac{1}{2}
\begin{pmatrix}
I_{r,s} & O  \\
O & I_{r,s} 
\end{pmatrix}$$
where  
$$I_{r,s} =
\begin{array}{c}
\begin{pmatrix}
  1 &  &  &  &  &  \\
   & \ddots &  &  &  &  \\
   &  & 1 &  &  &  \\
   &  &  & -1 &  &  \\
   &  &  &  & \ddots &  \\
   &  &  &  &  & -1 \\
\end{pmatrix} \\
\begin{array}{cc}
\underbrace{\hphantom{\begin{matrix} 1 & \cdots & 1 \end{matrix}}}_{r} & 
\underbrace{\hphantom{\begin{matrix} -1 & \cdots & -1 \end{matrix}}}_{s}
\end{array}
\end{array}.$$

Let $$\{e_i=v_i-\sqrt{-1}v_{i+r+s}:\, i=1,\cdots,r+s\}$$ be a basis of $V_+$.
Then the matrix $Q(e_i,e_j)$ under the basis $\vec{e}=(e_1,\cdots,e_{r+s})^T$ is 
$$Q_e =
\begin{pmatrix}
O & J_{r,s}  \\
J_{r,s} & O 
\end{pmatrix},$$
where
$$ 
J_{r,s} =
\begin{array}{c}
\begin{pmatrix}
   &  &  &  &  & 1 \\
   &  &  &  & \begin{sideways}$\ddots$\end{sideways} &  \\
   &  &  & 1 &  &  \\
   &  &  -1& &  &  \\
   &  \begin{sideways}$\ddots$\end{sideways}&  &  & &  \\
  -1 &  &  &  &  &  \\
\end{pmatrix} \\
\begin{array}{cc}
\underbrace{\hphantom{\begin{matrix} -1 & \cdots  & -1 \end{matrix}}}_{s} & 
\underbrace{\hphantom{\begin{matrix} 1 &\cdots  & 1 \end{matrix}}}_{r}
\end{array}
\end{array}
.$$

Let $$\mathcal{U}_{\mathbb F}(r,s)=\mathrm{Aut}_{\mathbb F}(V,Q)$$
be the automorphism group of the $\mathbb{Q}$-vector space $V$, which preserves the bilinear form $Q$ and commutes with the action \eqref{action of -1} of $\sqrt{-1}$.

Let $$\mathcal{U}(r,s)=\mathrm{Res}_{\mathbb F/\mathbb Q}\,\mathcal{U}_{\mathbb F}(r,s).$$ Then
$$\mathcal{U}(r,s) = \left\{\begin{pmatrix}
                       A & B\\
                       -B & A
                     \end{pmatrix}\in \mathrm{GL}(2(r+s),\mathbb Q)\left|\, \begin{array}{l}
                                                                    A^TI_{r,s}A+B^TI_{r,s} B=I_{r,s} \\
                                                                    A^T I_{r,s}B=B^TI_{r,s} A 
                                                                  \end{array}\right. \right\},$$
which is an algebraic group over $\mathbb Q$, and 
$$\mathcal{U}_{\mathbb F}(r,s)= \left\{C=A+\sqrt{-1}B\in \mathrm{GL}(r+s,\mathbb F)\left|\, \bar{C}^T I_{r,s}C=I_{r,s}\right. \right\}.$$                                                         
Then complexification $\mathcal{U}(r,s)_\C$ of $\mathcal{U}(r,s)$ is isomorphic to the classical group $U(r,s)$.

Define a decomposition on $V_{+,\C}$ by
\begin{eqnarray}
  V_{+,\C}&=&V_+^{2r,0}\oplus V_+^{2r-1,1}\oplus V_+^{2r-2,2}\oplus V_+^{2r-3,3}\oplus\cdots\oplus V_+^{2r-s+1,s-1}\oplus V_+^{2r-s,s} \label{Hodge r,s V_+}\\
  &&\oplus V_+^{2r-s-1,s+1}\oplus V_+^{2r-s-2,s+2}
  \oplus\cdots\oplus V_+^{1,2r-1}\oplus V_+^{0,2r}  \nonumber\\
  &=&\C\{e_1\}\oplus \C\{e_{r+1}\}\oplus \C\{e_2\}\oplus \C\{e_{r+2}\} \oplus \cdots\oplus  \C\{e_s\}\oplus \C\{e_{r+s}\}\nonumber \\
  && \oplus  \C\{e_{s+1}\}\oplus 0\oplus \cdots\oplus  \C\{e_{r}\}\oplus 0\nonumber
\end{eqnarray}
and define the decomposition $$V_{-,\C}=\oplus_{p+q=2r}V_-^{p,q}$$ by 
\begin{equation}\label{Hodge r,s V_-}
  V_-^{p,q}=\bar{V_+^{q,p}},\, \forall\, p+q=2r.
\end{equation}
Then 
\begin{equation}\label{Hodge r,s}
  V_\C=V_{+,\C}\oplus V_{-,\C}
\end{equation}
together with the decompositions \eqref{Hodge r,s V_+} and \eqref{Hodge r,s V_-} gives a polarized Hodge structure of weight $2r$ on $V_\C$.

According to Page 40 of \cite{GGK}, this polarized Hodge structure has Mumford-Tate group $$S\mathcal{U}(r,s)=\mathcal{U}(r,s)\cap \mathrm{Res}_{\mathbb F/\mathbb Q}\,\mathrm{SL}(2(r+s),\mathbb F),$$ whose complexification $S\mathcal{U}(r,s)_\C$ is isomorphic to the classical group $SU(r,s)$.

Since $V_{-,\C}=\bar{V_{+,\C}}$, the real subgroup $S\mathcal{U}(r,s)_\R$ of $S\mathcal{U}(r,s)_\C$ can be considered as the automorphism group of $V_{+,\C}$ preserving the Hermitian form 
$$H(\cdot,\cdot)|_{V_{+,\C}} \sim \frac{1}{2}I_{r,s}.$$
Then the Mumford-Tate domain $D$ is isomorphic to the flag domain $S\mathcal{U}(r,s)_\R/T$.

The maximal torus $T$ of $S\mathcal{U}(r,s)_\R$ is given by
$$\left\{h=\begin{pmatrix}
                       e^{2\pi\sqrt{-1}\theta_1} & &\\
                    &\ddots &\\
                    \\
                    &&e^{2\pi\sqrt{-1}\theta_{r+s}}
                     \end{pmatrix}:\, \sum_{i=1}^{r+s}\theta_i=2k \right\} ,$$
and the Cartan subgroup $\mf h=(\mathrm{Lie}\, T)_\C \cong \C^{r+s}$ with the basis $e_1,\cdots,e_{r+s}$. 
Then we can identify $$X=\mathrm{diag}(\theta_1,\cdots,\theta_{r+s})\in \mf h$$ with $$\theta= (\theta_1,\cdots,\theta_{r+s})=\theta_1e_1+\cdots+\theta_{r+s}e_{r+s}.$$

Let $e_1^*,\cdots,e_{r+s}^*$ be the dual basis of $e_1,\cdots,e_{r+s}$. Then 
$$[X,E_{ij}]=(\theta_i-\theta_j)E_{ij}=(e_i^*-e_j^*)(X)E_{ij},$$
where $E_{ij}$ is the matrix $(a_{kl})$ with $a_{kl}=1$ if $k=i,l=j$ and $a_{kl}=0$ otherwise.

The set of simple roots of $SU(r,s)$ is given by
\begin{eqnarray*}
  && \epsilon_1=e^*_1-e^*_2,\epsilon_2=e^*_2-e^*_3, \cdots,\epsilon_{r-1}=e^*_{r-1}-e^*_r\\
  && \epsilon_r=e^*_r-e^*_{r+1},\\
  && \epsilon_{r+1}=e^*_{r+1}-e^*_{r+2},\cdots, \epsilon_{r+s-1}=e^*_{r+s-1}-e^*_{r+s}
\end{eqnarray*}
while $\epsilon_1,\cdots,\epsilon_{r-1},\epsilon_{r+1},\cdots,\epsilon_{r+s-1}$ are compact and $\epsilon_r$ is non-compact.

The complex structure of the Mumford-Tate domain $D$ is given by 
\begin{equation}\label{g -k k}
\mf n_-=\bigoplus_{k>0}\mf g^{-k,k},
\end{equation}
where $$\mf g^{-k,k}=\{X\in \mf g:\, XV_+^{p,q}\subset V_+^{p-k,q+k},\, \forall\, p+q=2r\}.$$

The sets of positive roots $\Delta_+=\Delta_+^\c\cup \Delta_+^\nc$ corresponding to the complex structure \eqref{g -k k} of the Mumford-Tate domain $D$ are
\begin{eqnarray}
 \Delta_+^\c &=&\{e^*_i-e^*_j:\,1\le i<j\le r\}\cup\{e^*_{r+l}-e^*_{r+m}:\, 1\le l<m\le s\} \label{Delta c for SU};\\
 \Delta_+^\nc &=&\{e^*_{l}-e^*_{r+m}:\, 1\le l\le m\le s\}\cup \{e^*_{r+l}-e^*_{i}:\, 1\le l\le s,\, l<i\le r\}.\nonumber
\end{eqnarray}
Since 
$$ {\Delta'}_+^\nc =\{e^*_{l}-e^*_{r+m}:\, 1\le l\le m\le s\}\cup \{e^*_{i}-e^*_{r+l}:\, 1\le l\le s,\, l<i\le r\}$$
gives a complex structure $\mathbb B$ on the symmetric space $G_\R/K$,
we have that 
\begin{eqnarray}
 \Delta_+^\nca&=&\{e^*_{l}-e^*_{r+m}:\, 1\le l\le m\le s\} \label{Delta nc1 for SU}\\
  &=&\{e^*_{1}-e^*_{r+1},\cdots, e^*_{1}-e^*_{r+s};\,e^*_{2}-e^*_{r+2},\cdots,e^*_{2}-e^*_{r+s};\,\cdots e^*_{s}-e^*_{r+s}\},\nonumber\\
 \Delta_+^\ncb&=& \{e^*_{r+l}-e^*_{i}:\, 1\le l\le s,\, l<i\le r\} \label{Delta nc2 for SU}\\
 &=&\{e^*_{r+1}-e^*_{2},\cdots,e^*_{r+1}-e^*_{r};\,e^*_{r+2}-e^*_{3},\cdots,e^*_{r+2}-e^*_{r};\,\cdots;\nonumber\\
 && \,\,\, e^*_{r+s}-e^*_{s+1},\cdots,e^*_{r+s}-e^*_{r} \}.\nonumber
\end{eqnarray}

By direct computations, we have that 
\begin{eqnarray*}
  2\rho_\c&=& (r-1)e^*_1+(r-3)e^*_2+\cdots +(1-r)e^*_{r}\\&&+(s-1)e^*_{r+1}+(s-3)e^*_{r+2}+\cdots+(1-s)e^*_{r+s}; \\
  2\rho_\nc &=&se^*_1+(s-2)e^*_2+\cdots+(2-s)e^*_s+(-s)\left(e^*_{s+1}+\cdots e^*_r\right)\\&& + (r-2)e^*_{r+1}+(r-4)e^*_{r+2}+\cdots+(r-2s)e^*_{r+s};\\
  2(\rho_\nca-\rho_\ncb) &=&s(e^*_1+\cdots e^*_{r})-r(e^*_{r+1}+\cdots+e^*_{r+s}).
\end{eqnarray*} 
\end{example}

\begin{proposition}\label{SU s+1 s}
Let the notations be as in Example \eqref{U(r,s)}. Then $$\rho_\c=\rho_\nc$$ for $SU(r,s)$ if and only if $r=s+1$.

Therefore Corollary \ref{cup-products2} implies that on the compact and smooth quotient $X$ of the non-classical Mumford-Tate domain $D=S\mathcal{U}(s+1,s)_\R/T$, 
there exists a homomorphism of groups
\begin{equation*}
H^0(\Gamma\backslash\mathbb B,\omega_{\mathbb B}^{\otimes k/k_0})\times H^0(\Gamma\backslash \bar{\mathbb B}, \omega_{\bar{\mathbb B}}^{\otimes k/k_0}) \to H^{d}(X,L_{-\rho})^*,
\end{equation*}
for $k\ge N$.
\end{proposition}

\begin{example}[$Sp(2n,\C), n\ge 2$]\label{Sp(2n,C)}
Let $V$ be complex vector space of dimension $2n$, on which there exists an anti-symmetric bilinear form $Q$. Let $v_1,\cdots,v_{2n}$ be a basis of $V$ so that the matrix $Q(v_i,v_j)$ under the basis $v$ is
$$
J=
\begin{array}{c}
\begin{pmatrix}
   &  &  &  &  & 1 \\
   &  &  &  & \begin{sideways}$\ddots$\end{sideways} &  \\
   &  &  & 1 &  &  \\
   &  &  -1& &  &  \\
   &  \begin{sideways}$\ddots$\end{sideways}&  &  & &  \\
  -1 &  &  &  &  &  \\
\end{pmatrix} \\
\begin{array}{cc}
\underbrace{\hphantom{\begin{matrix} -1 & \cdots  & -1 \end{matrix}}}_{n} & 
\underbrace{\hphantom{\begin{matrix} 1 &\cdots  & 1 \end{matrix}}}_{n}
\end{array}
\end{array}
.$$

Define the complex conjugate $\sigma:\, V\to V$, $\sigma^{2}=-1$, by
$$\sigma v_{i}=\sqrt{-1}v_{2n+1-i},\, 1\le i\le n.$$
Then $\sigma v_{2n+1-i}=\sqrt{-1}v_{i}$. The complex conjugate defines a real subspace $V_{\R}$ of $V$ so that $V=V_{\R}\otimes_{\R}\C$.

Let 
\begin{eqnarray*}
V&=&V^{2n-1,0}\oplus V^{2n-2,1}\oplus V^{2n-3,2}\oplus V^{2n-4,3}\oplus \cdots\oplus V^{1,2n-2} \oplus V^{0,2n-1}\\
&=&\C\{v_{1}\}\oplus \C\{v_{2n-1}\}\oplus\C\{v_{3}\}\oplus \C\{v_{2n-3}\}\oplus\cdots\oplus\C\{v_{2}\}\oplus\C\{v_{2n}\}
\end{eqnarray*}
be a polarized Hodge structure of weight $2n-1$ on $V$ with $\dim_{\C}V^{p,q}=1$.

When $n\ge 2$, the period domain $D$ of all the polarized Hodge structures of weight $2n-1$ on $V$ with $\dim_{\C}V^{p,q}=1$ is a non-classical flag domain $D=Sp(2n,\R)/T$, where $T$ is the compact Cartan subgroup of $Sp(2n,\R)$.

The set of simple roots of $Sp(2n,\C)$ is given by
$$\epsilon_1=e^*_1-e^*_2,\cdots, \epsilon_{n-1}=e^*_{n-1}-e^*_n,\epsilon_{n}=2e^*_n.$$
where $\epsilon_1,\cdots,\epsilon_{n-1}$ are compact and $\epsilon_{n}$ is non-compact.

The sets of positive roots $\Delta_+=\Delta_+^\c\cup \Delta_+^\nc$ corresponding to the complex structure of the period domain $D$ are
\begin{eqnarray}
 \Delta_+^\c &=&\{(-1)^{i-1}(e^*_i-e^*_j):\, 1\le i<j\le n\} \label{Delta c for Sp};\\
 \Delta_+^\nc &=&\{(-1)^{i-1}(e^*_i+e^*_j):\, 1\le i \le j\le n\}.\nonumber
\end{eqnarray}
Since 
$$ {\Delta'}_+^\nc =\{e^*_i+e^*_j:\, 1\le i \le j\le n\}$$
gives a complex structure $\mathbb B$ on the symmetric space $G_\R/K$,
we have that 
\begin{eqnarray}
 \Delta_+^\nca&=&\{e^*_i+e^*_j:\, 1\le i \le j\le n, \, i\text{ odd} \} ,\label{Delta nc1 for Sp}\\
 \Delta_+^\ncb&=& \{-e^*_i-e^*_j:\, 1\le i \le j\le n, \, i\text{ even} \}.\label{Delta nc2 for Sp}
 \end{eqnarray}
 
We can compute directly that
$$\rho_\nc-\rho_\c=\left\{\begin{array}{l}e^*_1+e^*_3+\cdots+e^*_{2m-1},\, n=2m\\
e^*_1+e^*_3+\cdots+e^*_{2m+1},\, n=2m+1\end{array}\right. .$$
\end{example}

\begin{example}[$Sp(4,\C)$]
When $n=2$, we have 
\begin{eqnarray*}
 \Delta_+^\c &=&\{e^*_1-e^*_2\} ,\\
 \Delta_+^\nca&=&\{2e^*_1,e^*_1+e^*_2 \} ,\\
 \Delta_+^\ncb&=& \{-2e^*_2\},\\
 \rho_\nc-\rho_\c&=&e^*_1.
\end{eqnarray*}
We need to be careful here, since $(e^*_1+e^*_2, e^*_1-e^*_2)=0$, c.f. Remark \ref{beta alpga >0 remark} (2).

Let $\mu_{0}=a_{1}e^*_1+a_{2}e^*_2$ and $$\lambda_{0}=(1-a_{1})e^*_1-a_{2}e^*_2$$ so that $\mu_{0}+\lambda_{0}=\rho_\nc-\rho_\c$. Then \eqref{mu0 lambda_{0} alpha >=0} in Theorem \ref{cup-products} is equivalent to 
$$ a_{2}\le a_{1}\le a_{2}+1.$$

If $a_{1}=a_{2}=a$, then $\mu_{0}=a(e^*_1+e^*_2)$ and $$(\mu_{0}-(e^*_1+e^*_2),e^*_1-e^*_2)=0.$$ Hence \eqref{mu0 lambda_{0} beta <0} is not satisfied;
if $a_{1}=a_{2}+1=a+1$, then $\mu_{0}=e_{1}^{*}+a(e^*_1+e^*_2)$ and $$(\mu_{0}-(e^*_1+e^*_2),e^*_1-e^*_2)=1>0.$$ Hence \eqref{mu0 lambda_{0} beta <0} is neither satisfied. 

Therefore,  our method in this paper does not apply to the case of $Sp(4,\C)$.
\end{example}

\begin{example}[$Sp(6,\C)$]
When $n=3$, we have 
\begin{eqnarray*}
 \Delta_+^\c &=&\{e^*_1-e^*_2,e^*_1-e^*_3,-e^*_2+e^*_3\} ,\\
 \Delta_+^\nca&=&\{2e^*_1,e^*_1+e^*_2,e^*_1+e^*_3, 2e^*_3 \} ,\\
 \Delta_+^\ncb&=& \{-2e^*_2, -e^*_2-e^*_3\},\\
 \rho_\nc-\rho_\c&=&e^*_1+e^*_3.
\end{eqnarray*}

Let $\mu_{0}=0$ and $$\lambda_{0}=e^*_1+e^*_3\in \Delta_+^\nca$$ so that $\mu_{0}+\lambda_{0}=\rho_\nc-\rho_\c$.
Then \eqref{mu0 lambda_{0} alpha >=0} in Theorem \ref{cup-products} is automatically satisfied.
Note that \eqref{mu0 lambda_{0} beta <0} is also satisfied for $\mu_{0}=0$ due to Lemma \ref{beta alpga >0}.

For $\gamma=-2e^*_2$, let $\alpha_{\gamma}=e^*_1-e^*_2$ and then $$(\lambda_{0}-\gamma, \alpha_{\gamma})=(e^*_1+2e^*_2+e^*_3,e^*_1-e^*_2)=-1<0.$$
For $\gamma=-e^*_2-e^*_3$, let $\alpha_{\gamma}=e^*_1-e^*_3$ and then $$(\lambda_{0}-\gamma, \alpha_{\gamma})=(e^*_1+e^*_2+2e^*_3,e^*_1-e^*_3)=-1<0.$$
Therefore \eqref{mu0 lambda_{0} beta <0} is satisfied for $\lambda_0$.
\end{example}

From Theorem \ref{cup-products}, we have the following result.

\begin{proposition}
Let $X$ be the compact and smooth quotient of the period domain $D=Sp(6,\R)/T$. Then
there exists a homomorphism of groups
\begin{equation*}
H^0(\Gamma\backslash\mathbb B,\omega_{\mathbb B}^{\otimes k/k_0})\times H^0(\Gamma\backslash \bar{\mathbb B}, \omega_{\bar{\mathbb B}}^{\otimes k/k_0}\otimes L_{e^*_1+e^*_3}) \to H^{d}(X,L_{-\rho})^*,
\end{equation*}
for $k\ge N$.
\end{proposition}

We can show that when $n\ge 4$, we do not have similar cup-products with the method of this paper. More precisely we have the following proposition.

\begin{proposition}
For $Sp(2n,\C)$ with $n\ge 4$, there exists no $\mu_{0}+\lambda_{0}=\rho_\nc-\rho_\c$ with \eqref{mu0 lambda_{0} alpha >=0} and \eqref{mu0 lambda_{0} beta <0} satisfied.
\end{proposition}
\begin{proof}
\noindent {Case 1.} $n=2m$. Then $\rho_\nc-\rho_\c=e^*_1+e^*_3+\cdots+e^*_{2m-1}$.

Let $$\mu_{0}=\sum_{i=1}^{2m}a_{i}e_{i}^{*}.$$Then $\mu_{0}+\lambda_{0}=\rho_\nc-\rho_\c$ implies that 
$$\lambda_{0}=\sum_{i=1}^{m}\left[(1-a_{2i-1})e_{2i-1}^{*} +(-a_{2i})e_{2i}^{*}\right].$$

Now we consider the conditions \eqref{mu0 lambda_{0} alpha >=0} in Theorem \ref{cup-products}.
By taking $$\alpha=e_{2i-1}^{*}-e_{2i+1}^{*}\in \Delta_{+}^{\c},\, \, 1\le i\le m-1,$$ we have that
$$a_{1} =a_{3}=\cdots =a_{2m-1}=a.$$
Similarly by taking $$\alpha=e_{2i}^{*}-e_{2i+2}^{*}\in \Delta_{+}^{\c}, \, \, 1\le i\le m-1,$$ we have that 
$$a_{2} =a_{4}=\cdots =a_{2m}=b.$$
Finally we take $\alpha=e_{1}^{*}-e_{2}^{*}$ and have that 
$$b\le a \le b+1.$$
Hence we have two choices of $\mu_{0}$ and $\lambda_{0}$:
\begin{equation}\label{choice 1 of mu0 lambda0}
\begin{cases}
\mu_{0}=a(\sum_{i=1}^{2m}e_{i}^{*})\\ \lambda_{0}=\sum_{i=1}^{m}\left[(1-a)e_{2i-1}^{*} +(-a)e_{2i}^{*}\right]
\end{cases},
\end{equation}
and 
\begin{equation}\label{choice 2 of mu0 lambda0}
\begin{cases} \mu_{0}=\sum_{i=1}^{m}\left[(1+b)e_{2i-1}^{*} +be_{2i}^{*}\right]\\ \lambda_{0}=-b(\sum_{i=1}^{2m}e_{i}^{*})\end{cases} .
\end{equation}

In \eqref{choice 1 of mu0 lambda0}, $(\mu_{0},\Delta_{+}^{\c})=0$ and hence $\mu_{0}$ satisfies \eqref{mu0 lambda_{0} beta <0} in Theorem \ref{cup-products}. 

We claim that $\lambda_{0}$ in \eqref{choice 1 of mu0 lambda0} dost not satisfy \eqref{mu0 lambda_{0} beta <0}. In fact, let $$\gamma=-e_{2m-2}^{*}-e_{2m}^{*}\in \Delta_{+}^{\ncb}.$$ Then 
$$\lambda_{0}-\gamma=\sum_{i=1}^{m-2}\left[(1-a)e_{2i-1}^{*} +(-a)e_{2i}^{*}\right]+(1-a)\left(e_{2m-3}^{*}+e_{2m-2}^{*}+e_{2m-1}^{*}+e_{2m}^{*}\right),$$
and $(\lambda_{0}-\beta,\Delta_{+}^{\c})\ge 0$.

Similarly we can prove that $\mu_{0}$ in \eqref{choice 2 of mu0 lambda0} dost not satisfy \eqref{mu0 lambda_{0} beta <0}. In fact, we can take $$\beta=e_{2m-3}^{*}+e_{2m-1}^{*}\in \Delta_{+}^{\nca}.$$Then
$$\mu_{0}-\beta=\sum_{i=1}^{m-2}\left[(1+b)e_{2i-1}^{*} +be_{2i}^{*}\right]+b\left(e_{2m-3}^{*}+e_{2m-2}^{*}+e_{2m-1}^{*}+e_{2m}^{*}\right),$$
and $(\mu_{0}-\beta,\Delta_{+}^{\c})\ge 0$.

Therefore we have proved the proposition for $n=2m\ge 4$.

\noindent {Case 2.} $n=2m+1$. Then $$\rho_\nc-\rho_\c=e^*_1+e^*_3+\cdots+e^*_{2m+1}.$$

Let $$\mu_{0}=\sum_{i=1}^{2m+1}a_{i}e_{i}^{*},\, \lambda_{0}=\sum_{i=1}^{m}\left[(1-a_{2i-1})e_{2i-1}^{*} +(-a_{2i})e_{2i}^{*}\right]+(1-a_{2m+1})e_{2m+1}^{*}$$
such that $\mu_{0}+\lambda_{0}=\rho_\nc-\rho_\c$.

Similar to Case 1, we have that
\begin{eqnarray*}
&& a_{1} =a_{3}=\cdots =a_{2m+1}=a,\\
&& a_{2} =a_{4}=\cdots =a_{2m}=b,
\end{eqnarray*}
and that 
$$b\le a \le b+1.$$

Hence we have two choices of $\mu_{0}$ and $\lambda_{0}$:
\begin{equation}\label{choice 3 of mu0 lambda0}
\begin{cases}
\mu_{0}=a(\sum_{i=1}^{2m+1}e_{i}^{*})\\ \lambda_{0}=\sum_{i=1}^{m}\left[(1-a)e_{2i-1}^{*} +(-a)e_{2i}^{*}\right]+(1-a)e_{2m+1}^{*}
\end{cases},
\end{equation}
and 
\begin{equation}\label{choice 4 of mu0 lambda0}
\begin{cases} \mu_{0}=\sum_{i=1}^{m}\left[(1+b)e_{2i-1}^{*} +be_{2i}^{*}\right]+(1+b)e_{2m+1}^{*}\\ \lambda_{0}=-b(\sum_{i=1}^{2m+1}e_{i}^{*})\end{cases} .
\end{equation}

We take $$\gamma=-e_{2m-2}^{*}-e_{2m}^{*}\in \Delta_{+}^{\ncb}$$ and $$\beta=e_{2m-3}^{*}+e_{2m-1}^{*}\in \Delta_{+}^{\nca}.$$ 
Then
\eqref{choice 3 of mu0 lambda0} implies that
\begin{eqnarray*}
\lambda_{0}-\gamma&=&\sum_{i=1}^{m-2}\left[(1-a)e_{2i-1}^{*} +(-a)e_{2i}^{*}\right]+(1-a)\left(e_{2m-3}^{*}\right. \\
&&\left.+e_{2m-2}^{*}+e_{2m-1}^{*}+e_{2m}^{*}+e_{2m+1}^{*}\right),
\end{eqnarray*}
and \eqref{choice 4 of mu0 lambda0} implies that
\begin{eqnarray*}
\mu_{0}-\beta&=&\sum_{i=1}^{m-2}\left[(1+b)e_{2i-1}^{*} +be_{2i}^{*}\right]+b\left(e_{2m-3}^{*}\right.\\
&&\left.+e_{2m-2}^{*}+e_{2m-1}^{*}+e_{2m}^{*}+e_{2m+1}^{*}\right).
\end{eqnarray*}
Hence
$$(\lambda_{0}-\gamma, \Delta_{+}^{\c})\ge 0,\,(\mu_{0}-\beta,\Delta_{+}^{\c})\ge 0.$$
Therefore we have proved the proposition for $n=2m+1\ge 5$.
\end{proof}


\begin{thebibliography}{99}
\bibitem{BE}
R.~J.~Baston, and M.~G.~Eastwood,
\newblock{\em The Penrose transform: its interaction with representation theory}, 
\newblock{Courier Dover Publications}, (2016).

\bibitem{Bott57}
R.~Bott, 
\newblock{Homogeneous vector bundles}, 
\newblock{\em Annals of Math.}, \textbf{66} (1957), pp.~203--248.

\bibitem{BHH}
D.~Burns, S.~Halverscheid, and R.~Hind, 
\newblock{The geoemtry of Grauert tubes and complexification of symmetric spaces}, 
\newblock{\em Duke J. Math.}, \textbf{118} (2003), pp.~465--491.

\bibitem{C1}
H.~Carayol, 
\newblock{Limites dégénérées de séries discrètes, formes automorphes et variétés de
Griffiths-Schmid: le cas du groupe U(2, 1)}, \newblock{\em Compos. Math.}, \textbf{111} (1998), pp.~51--88.

\bibitem{C2}
H.~Carayol, 
\newblock{Quelques relations entre les cohomologies des variétés de Shimura et celles de
Griffiths-Schmid (cas du groupe SU(2, 1))}, \newblock{\em Compos. Math.}, \textbf{121} (2000), pp.~305--335.

\bibitem{C3}
H.~Carayol, \newblock{Cohomologie automorphe et compactifications partielles de certaines variétés de
Griffiths-Schmid}, \newblock{\em Compos. Math.}, \textbf{141} (2005), pp.~1081--1102.

\bibitem{CK07}
H.~Carayol and A.~W.~Knapp,
\newblock{Limits of Discrete Series with Infinitesimal Character Zero},
\newblock{\em Transactions of the American Mathematical Society}, Vol. \textbf{359}, No. 11 (2007), pp.~5611--5651.


\bibitem{CT}
J. ~Carlson and D. ~Toledo,
\newblock{Compact quotients of non-classical domains are not K\"ahler,}
\newblock{\em Contemporary Mathematics,} \textbf{608} (2014), pp.~51--57.

%
%
%

%

%
%

%


\bibitem{EGW}
M.~Eastwood, S.~Gindikin, and H.~Wong, 
\newblock{Holomorphic realizations of $\bar{\partial}$-cohomology and constructions of representations}, 
\newblock{\em J. Geom. Phys.}, \textbf{17}(3) (1995), pp.~231--244.

\bibitem{FHW}
G.~Fels, A.~Huckleberry and J.~A.~Wolf, 
\newblock{\em Cycle Spaces of Flag Domains, Progress in Mathematics},
\textbf{245}, \newblock{Birkh\"auser}, (2006).


%

%
%
%
%
%
%
%
%
%
%
\bibitem{GGK12}
M.~Green, P.~Griffiths and M.~Kerr,
\newblock {\em Mumford-Tate Groups and Domains: Their Geometry and Arithmetic}, 
\newblock {Annals of Math. Studies}, \textbf{183}, Princeton University Press, Princeton, NJ, (2012).

\bibitem{GGK}
M.~Green, P.~Griffiths and M.~Kerr,
\newblock {\em Hodge Theory, Complex Geometry, and Representation Theory},
\newblock {Regional Conference Series in Mathematics}, Volume \textbf{118}, American Mathematical Society, (2013).

\bibitem{GGK14}
M.~Green, P.~Griffiths and M.~Kerr,
Special values of automorphic cohomology classes, Vol.~\textbf{231}, no.~1088, American Mathematical Society, (2014).

\bibitem{GS}
P.~Griffiths and W.~Schmid,
\newblock{Locally homogeneous complex manifolds,}
\newblock{\em Acta Math.}, \textbf{123} (1969), pp.~253--302.



\bibitem{GRT}
P.~Griffiths, C.~Robles and D.~Toledo,
\newblock{Quotients of non-classical flag domains are not algebraic},
\newblock{\em Algebraic Geometry}, \textbf{1} (2014), pp.~1--13.




%
\bibitem{Hel}
S.~Helgason,
\newblock{\em Differential geometry, Lie groups, and symmetric spaces,}
\newblock{Academic Press, New York}, (1978).

%
%
%
\bibitem{Hum}
J.~Humphreys,
\newblock{\em Introduction to Lie Algebras and Representation Theory},
\newblock{Springer-Verlag New York}, (1972).
%
%
%






%
%

\bibitem{Ker14}
M.~Kerr,
\newblock {Cup products in automorphic cohomology: The case of $Sp_4$},
\newblock {\em Contemporary Mathematics}, \textbf{608} (2014), pp.~199--234.


\bibitem{knapp}
A.~W.~Knapp, 
\newblock {\em Lie groups beyond an introduction},
\newblock {Progress in Mathematics}, Vol.~\textbf{140}, Boston: Birkhäuser, (1996).


\bibitem{knapp86}
A.~W.~Knapp, 
\newblock {\em Representation Theory of Semisimple Groups: An Overview Based on Examples}, 
\newblock {Princeton Univ. Press, Princeton}, NJ, (1986).


%
%
%

\bibitem{KS3}
K.~Kodaira and D.~C.~Spencer,
\newblock {On Deformations of Complex Analytic Structures, III,}
\newblock {\em Annals of Mathematics,} Second Series, \textbf{71}(1) (1960), pp.~43--76.
%


%
%
%

%



%
%
%
%
%
%
%
%


\bibitem{LiuShen24}
K.~Liu and Y.~Shen,
\newblock{Geometry of non-classical period domains},
\newblock{\em arXiv:2405.16536,} (2024).

%
%
%
%
%
%
%
%
%
%
%

%
%
%
%
%
%
%


%
%

\bibitem{schmid67}
W.~Schmid, 
\newblock{Homogeneous complex manifolds and representations of semisimple Lie groups}, 
\newblock{Ph.D. dissertation}, Univ. California, Berkeley (1967); reprinted in 
\newblock{\em Representation Theory and Harmonic Analysis on Semisimple Lie Groups}, Math. Surveys and Monogr. \textbf{31}, Amer. Math. Soc. (1989), pp.~223--286.





\bibitem{schmid97}
W.~Schmid,
\newblock {Discrete series}, 
\newblock{\em Proc. Symp. Pure Math.}, \textbf{61} (1997), pp.~83--113.
%
%
%
%
%



%




%
%

%
%
%
%


\bibitem{Wolf64}
J. ~A.~Wolf,
\newblock {On the Classification of Hermitian Symmetric Spaces},
\newblock {\em Journal of Mathematics and Mechanics}, Vol.~\textbf{13}, No.~3 (1964), pp.~489--495


\bibitem{Wi1}
F.~Williams, 
\newblock {Discrete series multiplicities in $ L^2(\Gamma \backslash G) $ (II). Proof of Langlands conjecture}, 
\newblock {\em Amer. J. Math.}, \textbf{107} (1985), pp.~367--376.

\bibitem{Wi2} 
F.~Williams, 
\newblock {The $ n $-cohomology of limits of discrete series}, 
\newblock {\em J. Funct. Anal.}, \textbf{80} (1988), pp.~451--461.



\end{thebibliography}
\end{document}